\documentclass[a4paper,12pt]{amsart}
\usepackage[utf8]{inputenc}
\usepackage[T1]{fontenc}
\usepackage[english]{babel}
\usepackage{amssymb,amsthm,times}
\usepackage{xcolor}
\usepackage{hyperref}
\usepackage{cancel}
\usepackage{graphicx}
\usepackage{float}
\usepackage{epstopdf}
\usepackage[arrow, matrix, curve]{xy}
\usepackage{graphicx}

\newtheorem{Theo}{Theorem}[section]
\newtheorem{Prop}[Theo]{Proposition}
\newtheorem{Coro}[Theo]{Corollary}
\newtheorem{Lemm}[Theo]{Lemma}

\newtheorem{Obse}[Theo]{Observation}

\newtheorem{Rema}[Theo]{Remark}

\newcommand{\IDT}{\mathbb{T}^{\N}}
\newcommand{\rto}{\rightarrow}

\newcommand{\N}{\mathbb{N}}
\newcommand{\SSS}{\mathcal{S}}
\newcommand{\R}{\mathbb{R}}

\newcommand{\BB}{\mathcal{B}}
\newcommand{\primes}{\mathfrak{p}}

\newcommand{\spn}{\operatorname{span}}

\newcommand{\Real}{\operatorname{Re}}

\title{\bf Hardy spaces of vector-valued Dirichlet series}

\author{A. Defant}
\address{Institut f\"{u}r Mathematik. Universit\"{a}t Oldenburg. D-26111 Oldenburg (Germany).} \email{defant@mathematik.uni-oldenburg.de}

\author{A. P\'{e}rez}
\address{Departamento de Matem\'{a}ticas, Universidad de Murcia, Espinardo. 30100 Murcia (Spain).} \email{antonio.perez7@um.es}

\thanks{ The research of the second author was partially done during a stay in Oldenburg (Germany) under the support of a PhD fellowship of ``La Caixa Foundation'', and of the projects of MINECO/FEDER (MTM2014-57838-C2-1-P) and Fundación S\'{e}neca - Regi\'{o}n de Murcia (19368/PI/14).
}



\begin{document}

\begin{abstract}
Given a Banach space $X$   and $1 \leq p \leq \infty$, it is well known that the two Hardy spaces
$H_p(\mathbb{T},X)$ ($\mathbb{T}$ the torus) and $H_p(\mathbb{D},X)$ ($\mathbb{D}$ the disk) have to be distinguished carefully. This motivates us to
define and study
 two different types of Hardy spaces
$\mathcal{H}_p(X)$ and $\mathcal{H}^+_p(X)$ of Dirichlet series $\sum_n a_n n^{-s}$ with
coefficients in $X$. We
characterize them in  terms of summing operators as well as holomorphic functions in infinitely many variables, and prove that they coincide whenever $X$ has the analytic Radon-Nikod\'{y}m Property.
Consequences are, among others, a vector-valued version of the Brother's Riesz Theorem in the infinite-dimensional torus, and  an answer to the question when $\mathcal{H}_1(X^{\ast})$ is a dual space.

\end{abstract}

\maketitle

\section{Introduction}

Recent  years have seen a remarkable growth of interest in certain functional
analytic aspects of the theory of ordinary Dirichlet series $\sum_n a_n n^{-s}$.  We refer to the survey \cite{Hedenmalm} and the monograph \cite{QplusQ}, and highlight the articles \cite{Bayart1}  and \cite{HardySpaceDirich}. The aim of the present paper is to study properties of Hardy spaces of vector-valued Dirichlet series, this is, Dirichlet series with coefficients in a (complex) Banach space $X$. As  a sort of by-product, we will see that our vector-valued  point of view even  allows to obtain more information on  scalar-valued Dirichlet series.
But  before we explain more carefully what we intend to do, let us introduce some preliminaries.

Bohr established a formal one-to-one correspondence between Dirichlet series and formal power series in infinitely many variables with coefficients in $\mathbb{C}$.
This correspondence in a straight forward way extends to the vector-valued situation. Given a Banach space $X$, the  map, to which we will refer as \emph{Bohr transform}, is formally defined as
\[ \BB: \sum_{\alpha \in \mathbb{N}_{0}^{(\mathbb{N})}}{c_{\alpha} z^{\alpha}} \longleftrightarrow \sum_{n \in \N}{a_{n} n^{-s}}, \mbox{ where $a_{n}=c_{\alpha} \in X$ if $n=\primes^{\alpha}$}. \]
Here $\primes$ denotes the ordered sequence of prime numbers $\mathfrak{p}_{1} < \mathfrak{p}_{2} < \ldots$, the symbol $\N_{0}^{(\N)}$ stands for the set of all sequences $\alpha$ in $\N_{0}:=\N \cup \{0\}$ that are eventually null, and $z^{\alpha} := z_{1}^{\alpha_{1}} z_{2}^{\alpha_{2}} \cdot \ldots$ for any given sequence $z =(z_n)_{n \in \N}$ of scalars.

Consider the infinite-dimensional torus $\IDT$ with the Haar measure $d \omega$ (i.e. the normalized Lebesgue measure), and let $L_{p}(\IDT,X)$ be the space of $p$-Bochner integrable functions for $1 \leq p < \infty$ or the space of essentially-bounded measurable functions if $p=\infty$.  We know that every $f \in L_{p}(\IDT,X)$ is uniquely determined by its Fourier series
\[ f \equiv \sum_{\alpha \in \mathbb{Z}^{(\N)}}{\widehat{f}(\alpha) \omega^{\alpha}} \hspace{2mm} \mbox{ with } \hspace{2mm} \widehat{f}(\alpha) = \int_{\IDT}{f(\omega) \omega^{-\alpha} \: d \omega}\,, \]
 where $\mathbb{Z}^{(\N)}$ stands for the set of all sequences of integers that are eventually zero. We denote by $H_{p}(\IDT, X)$ the closed subspace of $L_{p}(\IDT, X)$ consisting of those functions $f$ for which $\widehat{f}(\alpha) = 0$ whenever $\alpha \in \mathbb{Z}^{(\N)} \setminus \N_{0}^{(\N)}$.

 We define $\mathcal{H}_p(X)$ as the image of $H_p(\IDT,X)$ through Bohr's transform,
  \[
  \mathcal{H}_p(X) \equiv H_p(\IDT,X)
  \]
    endowed with the norm that makes this identification an onto isometry between both spaces. It follows from this definition some features like the fact that if a sequence of Dirichlet series $(D^{N})_{N}$ converges in $\mathcal{H}_{p}(X)$ to some $D$, then the sequence coefficients $a_{n}(D^{N})$ converges to $a_{n}(D)$ for each $n \in \N$; or the denseness of the set of all Dirichlet polynomials $D = \sum_{n=1}^{N}{a_{n} n^{-s}}$, $N \in \N$ in $\mathcal{H}_{p}(X)$ for $1 \leq p < + \infty$, as so do the set of trigonometric polynomials in $H_{p}(\IDT,X)$. Moreover, the $\mathcal{H}_{p}(X)$-norm of such Dirichlet polynomials $D$ is given by the following two
intrinsic  formulas (see e.g. \cite{QplusQ}): For $1 \leq p < \infty$ we have
\begin{align}
\label{EQUA:pNormSeries}
\begin{split}
\| D\|_{\mathcal{H}_{p}(X)}
&
=
 \Big(\int_{\mathbb{T}^{\mathbb{N}}} \Big\|\sum_{\alpha:1 \leq p^\alpha \leq N} a_{p^\alpha} \omega^\alpha  \Big\|_X^p d \omega\Big)^{1/p}
 \\[1ex]&
=
\Big(\lim_{R\rightarrow \infty}{\frac{1}{2R}{ \int_{-R}^{R}{\Big\| \sum_{n=1}^{N}{a_{n} \frac{1}{n^{it} }}\Big\|_X^{p} \: d t}}}\Big)^{\frac{1}{p}},
\end{split}
\end{align}
and for $p=\infty$
\begin{align}
\label{EQUA:supremumNormSeries}
\begin{split}
\| D\|_{\mathcal{H}_{\infty}(X)}
&
=  \sup_{z \in \mathbb{T}^{\mathbb{N}}} \Big\|\sum_{\alpha:1 \leq p^\alpha \leq N} a_{p^\alpha} z^\alpha  \Big\|_X
=  \sup_{z \in \mathbb{D}^{\mathbb{N}}} \Big\|\sum_{\alpha:1 \leq p^\alpha \leq N} a_{p^\alpha} z^\alpha  \Big\|_X
\\[1ex]
&
= \sup_{t \in \mathbb{R}}{\Big\| \sum_{n=1}^{N}{a_{n} \frac{1}{n^{it} }}\Big\|_X}
= \sup_{\Real{(s)} \geq 0}{\Big\| \sum_{n=1}^{N}{a_{n} \frac{1}{n^{s} }}\Big\|_X}
\,.
\end{split}
\end{align}

 The  Banach space $H_p(\IDT,X)$ shares many features with the classical Hardy space $H_{p}(\mathbb{T},X)$ (defined similarly) as we point out in section \ref{sec:spaceProperties}.
 Given $1 \leq p < \infty$,
the Banach space $H_p(\mathbb{D},X)$   consists of all holomorphic functions $F:\mathbb{D} \rightarrow X$ such that
\[ \| F\|_{H_p(\mathbb{D},X)} :=\sup_{0 < r < 1}{\Big(\int_{\mathbb{T}}{\| F(r \omega) \|_X^{p} \: d \omega}\Big)^{p}} < \infty. \]
Recall that for  $p = \infty$ the Hardy space $H_\infty(\mathbb{D},X)$ is the Banach space of all bounded and holomorphic functions $F:\mathbb{D} \rightarrow X$ together with the supremum norm.
It is well-known that $H_p(\mathbb{D},X)$  contains $H_p(\mathbb{T},X)$ isometrically  by means of the Poison Kernel $K$: given $f \in H_p(\mathbb{T},X)$ the function
\begin{equation}
\label{equa:PoissonRepresentation}
F(z) = \int_{\mathbb{T}}{f(\omega) K(\omega, z) \: d \omega} \: \mbox{ where } \: K(\omega, z) := \frac{|\omega|^{2} - |z|^2}{|\omega - z|^2},
\end{equation}
belongs to $H_p(\mathbb{D},X)$ and $\|F \|_{H_p(\mathbb{D},X)} = \| f\|_{H_p(\mathbb{T},X)}$. In the scalar case this leads to an onto isometry --  however in the vector-valued case $H_p(\mathbb{D},X)$ and $H_p(\mathbb{T},X)$ are in general different
(see also section \ref{sect:ARNP}).

\noindent For each $\sigma \in \mathbb{R}$ denote
$$\mathbb{C}_\sigma:= \{ s \in \mathbb{C} \colon \Real{(s)} > \sigma \}\,.$$
Following the previous philosophy we define $\mathfrak{D}_{\infty}(X)$ as the space of all Dirichlet series $D=\sum_n a_n n^{-s}$ with coefficients $a_n \in X$ which converge on the half space $\mathbb{C}_0$ and define a bounded and (then automatically) holomorphic function $D(s)$ on that half-space. The next result of Bohr is fundamental for everything following (details of its proof are given in Remark \ref{Rema:proofBiebersau}).

\begin{Theo} \label{Biebersau}
 Every Dirichlet series $\sum_n a_n n^{-s}$  in a Banach space $X$ which has a limit function $D(s)$ extending to a bounded and holomorphic functions on $\mathbb{C}_0$, converges uniformly on all half spaces $\mathbb{C}_\varepsilon$ where $\varepsilon >0$\,.
\end{Theo}

 Based on this result it is quite simple to show that  $\mathfrak{D}_{\infty}(X)$
 together with the supremum norm
 $$\| D\|_{\mathfrak{D}_{\infty}(X)} := \sup_{s \in \mathbb{C}_{0}}{\| D(s)\|_X}$$ forms a Banach space. It is by now well-known that $$\mathcal{H}_\infty(\mathbb{C}) = \mathfrak{D}_{\infty}(\mathbb{C})$$
 (see e.g. \cite{HardySpaceDirich} or \cite{QplusQ} ).
 However, in the case $1 \leq p < \infty$ a  Dirichlet series $D \in \mathcal{H}_{p}(\mathbb{C})$ in general only defines a holomorphic function on $\mathbb{C}_{1/2}$ (see e.g. \cite{Bayart1}).

  How to overcome this ``gap'' between $1 \leq p < \infty$ and $p =\infty$ at least partly? We in a first step consider  for every Dirichlet series $D = \sum_{n}{a_{n} n^{-s}}$
   with coefficients in $X$ its translation by some $z \in \mathbb{C}$:
\[
D_z:= \sum_n \frac{a_n}{n^z} n^{-s}\,.
\]
And then in a second step,  we  define for  $1 \leq p \leq  \infty$ the Banach space
      $$\mathcal{H}^{+}_{p}(X)$$
      of all Dirichlet series $D$ in $X$ such that  $D_\varepsilon \in \mathcal{H}_p(X)$ for each $\varepsilon>0$ and   moreover
    \begin{align*}\label{dp}
 \|D\|_{\mathcal{H}^{+}_{p}(X)} := \sup_{\varepsilon > 0}{\|D_{\varepsilon}\|_{\mathcal{H}_{p}(X)}} < \infty;
 \end{align*}
in contrast to $\mathfrak{D}_{\infty}(X)$, the completeness of $\mathcal{H}^{+}_{p}(X)$ is obvious.
Note first that
\begin{equation} \label{note}
\mathfrak{D}_\infty(X) \subset \mathcal{H}^{+}_{\infty}(X)
\end{equation}
and that the inclusion is even an isometry; indeed, if $D \in \mathfrak{D}_\infty(X)$ and $\varepsilon >0$, then by Bohr's fundamental Theorem \ref{Biebersau}
the sequence $D^N_\varepsilon = \sum_{n \leq N} \frac{a_n}{n^\varepsilon } n^{-s}\,, N \in \mathbb{N}$ converges to $D_\varepsilon$ uniformly on $\mathbb{C}_0$. But  by \eqref{EQUA:supremumNormSeries} we have
\begin{align*}
\| D^N_\varepsilon\|_{\mathfrak{D}_{\infty}(X)}
&
= \sup_{t \in \mathbb{R}} \sup_{\delta >0}{\Big\| \sum_{n=1}^{N}{\frac{a_{n}}{n^\varepsilon} \frac{1}{n^{\delta+it} }}\Big\|_X}
\\
&
= \sup_{t \in \mathbb{R}}{\Big\| \sum_{n=1}^{N}{\frac{a_{n}}{n^\varepsilon} \frac{1}{n^{it} }}\Big\|_X}
 = \|D^N_\varepsilon\|_{\mathcal{H}_{\infty}(X)}\,.
\end{align*}
Hence the Cauchy sequence $(D^N_\varepsilon)_N$ in $\mathcal{H}^+_\infty(X)$ converges to some Dirichlet series which by the uniqueness of the coefficients
must be $D_\varepsilon$. Moreover, this implies that
$\| D_\varepsilon\|_{\mathfrak{D}_{\infty}(X)} =  \|D_\varepsilon\|_{\mathcal{H}_{\infty}(X)}$\,.
This gives that  $D \in \mathcal{H}^+_\infty(X)$ and $\| D\|_{\mathfrak{D}_{\infty}(X)}  = \|D\|_{\mathcal{H}^+_{\infty}(X)}$
\,. \qed\\

It is a bit more challenging to show that even
\[
\mathcal{H}^{+}_{\infty}(X) = \mathfrak{D}_\infty(X)\,,
\]
and  that for all $1 \leq p \leq \infty$  the  following  inclusions
\begin{equation*}\label{HpinDp }
  \mathcal{H}_{p}(X) \hookrightarrow \mathcal{H}^{+}_{p}(X)\,,
\end{equation*}
hold true and are isometric (see Corollary \ref{infty=infty} and Theorem \ref{Rema1}).
In the scalar case we will see  that even
 \begin{equation*}\label{scalarcase}
  \mathcal{H}_{p}(\mathbb{C})\equiv \mathcal{H}^{+}_{p}(\mathbb{C})\,,
\end{equation*}
however in the vector-valued case it will turn out that this in general is not true which  motivates the following two questions:

\vspace{1mm}
\noindent \textbf{Q1.} \emph{
For which Banach spaces $X$ do we have that $\mathcal{H}_p(X)$ and $\mathcal{H}^{+}_p(X)$ coincide?
}\\[-3.5mm]

In section \ref{sect:ARNP} we prove that $\mathcal{H}_p(X)$ and $\mathcal{H}^{+}_p(X)$ coincide for every
$1 \leq p \leq \infty$ if and only if they coincide for some $1 \leq p \leq \infty$, which in turn is equivalent to the fact that the Banach space $X$ has the so-called analytic Radon-Nikod\'{y}m Property, see Theorem \ref{Theo:ARNPimpliesOnto}. The proof relies on a representation result given in Lemma \ref{Prop:representationHolomorphic}, and a theorem due to Dowling \cite{DowlingARNPLebesgue}.

\vspace{1mm}
\noindent \textbf{Q2.} \emph{
Can we, for an arbitrary Banach space $X$, identify $\mathcal{H}^{+}_p(X)$ with some ``better-known'' Banach space ``bigger'' than $\mathcal{H}_p(X)$?
}\\[-3.5mm]

We deal with this question in section \ref{sect:operatorSpaces}. The answer is inspired on results of Blasco \cite{BlascoPositivepSumming}, \cite{BlascoBoundaryValues}.  As usual, we denote the space of all bounded and linear operators between two Banach spaces $E$ and $X$ by $\mathcal{L}(E,X)$ endowed with the operator norm $\| \cdot \|$.
The following notion (see \cite[Chapter IV, Section \S 3]{ScaheferBanachLattice}, \cite[Definition 1, p. 275]{BlascoPositivepSumming}) is crucial for our purposes:\\

\vspace{-1mm}

 \noindent  An operator $T \in \mathcal{L}(E,X)$, where $E$ is a Banach lattice, is said to be \emph{cone absolutely summing} if there is a constant $C > 0$ such that for every choice of finitely many positive elements $f_{1}, \ldots, f_{n} \in E$, we have that
\begin{equation*}
\label{equa:positive1SummingOperatorsCondition}
\sum_{i=1}^{n}{\| Tf_{i}\|} \leq C \: \sup_{\xi^\ast \in B_{E^\ast}}{ \sum_{i=1}^{n}{|\xi^\ast f_i|}}
\end{equation*}
(here $E^\ast$ denotes the dual of $E$ and $B_{E^\ast}$ its unit ball).
We denote by $\Lambda(E,X)$ the space of all cone absolutely summing operators and endow it with the norm $\| \cdot \|_{\Lambda}$ which is the infimum of all constants $C > 0$ satisfying \eqref{equa:positive1SummingOperatorsCondition}. It turns out that this is indeed a Banach space with $\| \cdot \| \leq \| \cdot \|_{\Lambda}$.\\

 Denote by $1 \leq p^\ast \leq \infty$ the conjugate of $p$, i.e. $1/p + 1/p^\ast = 1$ with the convention $1/\infty := 0$. Let us agree to write
 $$  E_{p}(\Omega,X):= L_{p}(\mu, X )\,,$$
 when $(\Omega, \Sigma,\mu)$ is a probablity space and $1\leq p < + \infty$, and
$$E_{\infty}(\Omega, X):= C(\Omega, X)\,,$$
whenever $\Omega$ is a compact Hausdorff space.
 Endowed with their natural norms both spaces form Banach spaces. If $X = \mathbb{C}$, we simply write
 \[ E_{p}(\Omega) := E_{p}(\Omega, \mathbb{C}).\]
 The following isometric equalities hold (see  \cite{BlascoPositivepSumming}):
 \begin{equation}
 \begin{split} \label{equa:BlascoRepresentation}
 &
 \big(\Lambda(E_{1}(\Omega),X)\| \cdot \|_{\Lambda}\big) = \big(\mathcal{L}(E_{1}(\Omega),X),\| \cdot \|\big)
 \\
 &
 \big(\Lambda(E_{\infty}(\Omega),X)\| \cdot \|_{\Lambda}\big) =\big(\Pi_{1}(E_{\infty}(\Omega),X), \pi_1(\cdot)\big)\,,
 \end{split}
 \end{equation}
 where $\pi_1(\cdot)$ as usual stands for the $1$-summing norm.

  \noindent Notice that the set of trigonometric polynomials is dense in $E_{p^\ast}(\IDT)$ for every $1 \leq p^\ast \leq \infty$, so each $T \in \mathcal{L}(E_{p^\ast}(\IDT),X)$ is uniquely determined by its ``Fourier coefficients'' $T(\omega^{-\alpha})$ for $\alpha \in \mathbb{Z}^{(\N)}$. We denote by $\mathcal{L}^{+}(E_{p^\ast}(\IDT),X)$ the subset of all operators $T \in \mathcal{L}(E_{p^\ast}(\IDT),X)$ with $T(\omega^{-\alpha})=0$ whenever $\alpha \in \mathbb{Z}^{(\N)} \setminus \N_{0}^{(\N)}$. In this sense, each $T \in L^{+}(E_{p^\ast}(\IDT),X)$ is uniquely given  by its Fourier series
\begin{equation}
\label{equa:FourierSeriesOperator}
\sum_{\alpha \in \N_{0}^{(\N)}}{T(\omega^{-\alpha})} z^{\alpha}.
\end{equation}
Define $$\Lambda^+\left(E_{p^\ast}(\IDT),X\right) := \Lambda\left(E_{p^\ast}(\IDT),X\right) \cap \mathcal{L}^+\left(E_{p^\ast}(\IDT),X\right)\,,$$
 which is clearly a closed subspace of \mbox{$\left(\Lambda(E_{p^\ast}(\IDT),X), \| \cdot \|_{\Lambda}\right)$}. The main result of section \ref{sect:operatorSpaces} is Theorem \ref{Theo:DirichletOperators} which states that the Bohr transform $ \mathcal{B}$ establishes an onto isometry:
\[
\left(\Lambda^+\left(E_{p^\ast}(\IDT),X\right), \| \cdot \|_{\Lambda} \right)  \equiv \left(\mathcal{H}^{+}_{p}(X), \| \cdot \|_{\mathcal{H}^{+}_{p}(X)}\right).
 \]
In section \ref{sec:holomorphic} we characterize  the Hardy spaces $\mathcal{H}^{+}_{p}(X)$ of Dirichlet series  in terms of certain spaces of holomorphic functions on  certain Reinhardt domains. This generalizes  results from  \cite{MultipliersMonomialSets} and \cite{HardySpaceDirich}. Finally, in section \ref{sec:equivalencesARNP} we combine results of the previous sections to study, among other things, the validity of a vector-valued version of the Brother's Riesz Theorem on the infinite-dimensional torus (Theorem \ref{Theo:analyticMeasures}). Moreover, we characterize when $\mathcal{H}_{p}(X^{\ast})$ is a dual Banach space, or when
$ \mathcal{H}_{p}(X^{\ast}) \equiv E_{p}(\IDT, X)^{\ast}_{+} $,  the  subspace of the dual of $E_{p}(\IDT,X)$ of those functionals which vanish on elements of the form $\omega^{-\alpha} \otimes x$ for every $\alpha \in \mathbb{Z}^{(\mathbb{N})} \setminus \mathbb{N}_0^{(\mathbb{N})}$ and $x \in X$
(Theorem \ref{Coro:H1dual}).

\section{Embeddings}
\label{sec:spaceProperties}

\noindent
The aim of the present section is to show the following embedding theorem.

\begin{Theo} \label{Rema1}
Let $1 \leq p \leq \infty$. Then  $\mathcal{H}_{p}(X) \subset \mathcal{H}^{+}_{p}(X)$, and the inclusion
is isometric.
\end{Theo}

The proof is postponed until the end of the section, once we develop the auxiliary results we need, which are of independent interest and will be useful in other sections. We start with the following simple lemma which is an obvious consequence
of the rotation invariance of the Lebesgue measure on $\mathbb{T}^{\mathbb{N}}$.
\begin{Lemm}
\label{Lemm:RotationCoefficients}
Let
$1 \leq p \leq   \infty$
and  $D=\sum_{n}{a_{n} n^{-s}} \in \mathcal{H}_{p}(X)$. Then for each $\theta \in \IDT$ we have that
 \[
 D^{\theta} := \sum_{n}{b_n n^{-s}} \in \mathcal{H}_{p}(X) \: \mbox{ and  $\|D^{\theta}\|_{\mathcal{H}_p(X)} = \| D\|_{\mathcal{H}_p(X)}$\,,}
 \]
 where $b_{n}:=a_{\primes^{\alpha}} \theta^{\alpha}$ for $n = \primes^\alpha$.
\end{Lemm}

\vspace{2mm}

For the next result we need a bit more of notation. We denote by $H_{\infty}(\mathbb{C}_0, X)$ the Banach space of all bounded and holomorphic functions from $\mathbb{C}_{0}$ to $X$, endowed with the supremum norm, and define its closed subspace $$\mathbb{A}(\mathbb{C}_{0}, X):= C(\overline{\mathbb{C}_0},X) \cap H_{\infty}(\mathbb{C}_0, X)\,.$$

\begin{Prop}
\label{Prop:representationHolomorphic}
Given  $1 \leq p <\infty$  and a Banach space $X$,
the mapping
\begin{align} \label{embedding}
 \mathcal{H}_{p}(X) \hookrightarrow \mathbb{A}(\mathbb{C}_0,\mathcal{H}_{p}(X))\,,\,\,\, D
 \mapsto [z \mapsto D_z]
\end{align}
is a well-defined isometric embedding. Moreover, for each $D \in \mathcal{H}_{p}(X)$
   \vspace{1mm}
\begin{enumerate}
\item[(i)] $\| D_{\varepsilon + it}\|_{\mathcal{H}_{p}(X)} = \| D_{\varepsilon}\|_{\mathcal{H}_p(X)}$ for each
$\varepsilon \ge0$ and $t \in \mathbb{R}$\,,\vspace{2mm}
\item[(ii)] the map $\varepsilon \mapsto \| D_{\varepsilon}\|_{\mathcal{H}_p(X)}$ is decreasing on $[0, \infty),$\vspace{2mm}
\item[(iii)] and $\text{$\mathcal{H}_p(X)$-$\lim_{\varepsilon \rightarrow 0^+}$}\, D_\varepsilon= D$\,.
\end{enumerate}
\end{Prop}

\begin{proof}
We show in a first step that the mapping in \eqref{embedding}, restricted
to the subspace of all Dirichlet polynomials in $\mathcal{H}_p(X)$, defines an isometry satisfying  (i)-(ii).
Take a Dirichlet polynomial $D= \sum_{n=1}^{N}{a_{n} n^{-s}}$.  By Lemma \ref{Lemm:RotationCoefficients} for each $\varepsilon \ge 0$ and each $t \in \mathbb{R}$
$$\| D_{\varepsilon + it}\|_{\mathcal{H}_{p}(X)} = \| D_\varepsilon\|_{\mathcal{H}_{p}(X)}\,.$$  The last equality, together with the (distinguished) maximum modulus principle applied to the function
$F \in \mathbb{A}(\mathbb{C}_0,\mathcal{H}_p(X))\,,\,F(z)=D_z $, yields that for each $\varepsilon \geq 0$ we have
\begin{equation*}\label{zero}
 \sup_{\Real{(z)} \geq \varepsilon}{\| F(z)\|_{\mathcal{H}_{p}(X)}} = \sup_{t \in \R}{\| F(\varepsilon + it)\|_{\mathcal{H}_{p}(X)}} = \| D_{\varepsilon}\|_{\mathcal{H}_{p}(X)}\,.
  \end{equation*}
Hence, the statement for Dirichlet polynomials is proved. Take now an  arbitrary Dirichlet series
 $D  \in \mathcal{H}_{p}(X)$. Then we can find a sequence of Dirichlet polynomials $(D^{N})_{N \in \N}$ which in $\mathcal{H}_{p}(X)$ is norm-convergent to $D$. By the previous discussion the functions
  \[
 F_N \in \mathbb{A}(\mathbb{C}_0,\mathcal{H}_p(X))\,,\,F_N(z) = D^N_z
  \]
 form a Cauchy sequence in $\mathbb{A}(\mathbb{C}_0,\mathcal{H}_p(X))$. Denote its limit in
 $\mathbb{A}(\mathbb{C}_0,\mathcal{H}_p(X))$ by $F$. Given $z \in \mathbb{C}_0$, we show  that
 \[
 D_z = F(z)  \in \mathcal{H}_p(X),
 \]
i.e. $a_n( D_z) = a_n(F(z))$ for each $n$. Indeed, we have that
$$\text{$\mathcal{H}_p(X)$-$\lim_N$}\, D^N_z= \text{$\mathcal{H}_p(X)$-$\lim_N$}\, F_N(z)=F(z)\,,$$
and as a consequence for each $n$
  \[
  a_n( D_z)=a_n(D)n^{-z} =\text{$X$-$\lim_N$} \,a_n(D^N)n^{-z}= a_n(F(z))\,.
  \]
Finally, the statements (i) and (ii) are simple consequence since we already proved them for polynomials; and (iii), being evident for polynomials, follows from the continuity of the embedding in \eqref{embedding}.
\end{proof}

\vspace{2mm}

\begin{Theo}
\label{Theo:representationHolomorphic2}
Given  $1 \leq p < +\infty$  and a Banach space $X$,
the mapping
\begin{align} \label{embedding2}
 \mathcal{H}_{p}^{+}(X) \hookrightarrow \mathfrak{D}_{\infty}(\mathcal{H}_{p}(X))\,,\,\,\, \sum_{n}{a_{n} n^{-s}} \mapsto \sum_{n}{(a_{n} n^{-s}) n^{-z}}
\end{align}
is a well-defined isometric embedding.
\end{Theo}

\begin{proof}
Let $D = \sum_{n}{a_{n} n^{-s}}$ be a Dirichlet series in  $\mathcal{H}^{+}_{p}(X)$. We are going to  show that the function
\begin{align}
F: \mathbb{C}_0 \rightarrow \mathcal{H}_p(X)\,,\, F(z) = \sum_n \frac{a_n}{n^z} n^{-s}
\end{align}
is  well-defined, holomorphic and bounded. Indeed, for each $\varepsilon >0$ we by assumption have that $D_\varepsilon \in \mathcal{H}_p(X)$, and hence by Lemma \ref{Prop:representationHolomorphic}
that
\[
F_\varepsilon: \mathbb{C}_0 \rightarrow \mathcal{H}_p(X)\,,\, F_\varepsilon(z) = \sum_n \frac{a_n}{n^{\varepsilon +z}} n^{-s}
\]
is well-defined, holomorphic and bounded with
$\|F_\varepsilon\|_\infty = \|D_\varepsilon\|_{\mathcal{H}_p(X)}$. Therefore $F$ is well-defined, holomorphic and bounded with $$\| F\|_{\infty} = \sup_{\varepsilon > 0}{\| F_{\varepsilon}\|_{\infty}} = \sup_{\varepsilon > 0}{\| D_{\varepsilon}\|_{\mathcal{H}_{p}(X)}} = \| D\|_{\mathcal{H}_{p}^{+}(X)}\,.$$
Notice that for $\Real{(z)} > 1$, the Dirichlet series $\sum_{n}{(a_{n}n^{-s})n^{-z}}$ with coefficients $a_{n}n^{-s} \in \mathcal{H}_{p}(X)$ is absolutely convergent. Therefore, $F(z)$ extends the limit function of $\sum_{n}{(a_{n}n^{-s})n^{-z}}$ to a bounded and holomorphic funtion on $\mathbb{C}_{0}$, and using now Theorem \ref{Biebersau} we conclude that for each $z \in \mathbb{C}_{0}$
\[
F(z)=
\text{$\mathcal{H}_p(X)$-$\lim_N$} \sum_{n=1}^N \frac{a_n}{n^z} n^{-s}
=
\text{$\mathcal{H}_p(X)$-$\lim_N$} \sum_{n=1}^N (a_n n^{-s}) \frac{1}{n^z}
\,,
\]
which means that $F \in \mathfrak{D}_{\infty}(\mathcal{H}_{p}(X))$.
\end{proof}

In order to tackle the proof of Theorem \ref{Rema1} in the case $p=+\infty$, we require the next result.

\begin{Lemm}
\label{Lemm:pnormsLimit}
If $D$ is a Dirichlet series that belongs to $\mathcal{H}_{p}^{+}(X)$ for every $p \in [1, + \infty)$, then
\[ \| D\|_{\mathcal{H}_{\infty}^{+}(X)} = \lim_{p \rightarrow + \infty}{\| D\|_{\mathcal{H}_{p}^{+}(X)}} \in [0, +\infty].\]
\end{Lemm}

\begin{proof}
It is a well-known fact that if a measurable function $f: \mathbb{T}^{\N} \rightarrow X$ belongs to $L_{p}(\mathbb{T}^{\N},X)$ for every $p \in [1, +\infty)$, then $\| f\|_{\infty} = \lim_{p \rightarrow + \infty}{\| f\|_{p}}$. Applying this, we obtain that
\begin{align*}
\| D\|_{\mathcal{H}_{\infty}^{+}(X)} & = \sup_{\varepsilon > 0}{\| D_{\varepsilon}\|_{\mathcal{H}_{\infty}(X)}} = \sup_{\varepsilon > 0}{\lim_{p \rightarrow + \infty}\| D_{\varepsilon}\|_{\mathcal{H}_{p}(X)}}\\
& \leq \liminf_{p \rightarrow \infty}{\sup_{\varepsilon > 0}{\| D_{\varepsilon}\|_{\mathcal{H}_{p}X)}}} = \liminf_{p \rightarrow \infty}{\| D\|_{\mathcal{H}_{p}^{+}(X)}}\\
& \leq \limsup_{p \rightarrow \infty}{\| D\|_{\mathcal{H}_{p}^{+}(X)}} \leq \| D\|_{\mathcal{H}_{\infty}^{+}(X)}.
\end{align*}
\end{proof}

\begin{proof}[Proof of Theorem \ref{Rema1}]
In the case $1 \leq p < + \infty$, this is obviously a consequence of Proposition \ref{Prop:representationHolomorphic},
(ii) and (iii). If $p=+\infty$, we can use the previous case and apply Lemma \ref{Lemm:pnormsLimit} to get
\[ \| D\|_{\mathcal{H}^{+}_{\infty}(X)} = \lim_{p \rightarrow + \infty}{\| D\|_{\mathcal{H}^{+}_{p}(X)}} = \lim_{p \rightarrow + \infty}{\| D\|_{\mathcal{H}_{p}(X)}} = \| D\|_{\mathcal{H}_{\infty}(X)}. \]
\end{proof}

\section{Approximation}

\noindent We need the following quantitative version of Bohr's fundamental theorem (see Theorem \ref{Biebersau}). The proof follows exactly the lines from the scalar case that can be found in \cite{Queffelec1} or \cite{QplusQ}.

\begin{Theo}
\label{old}
There exists a constant $C > 0$ such that  every Dirichlet series $D = \sum_{n}{a_{n} n^{-s}} \in \mathfrak{D}_{\infty}(X)$  we have that  for each $x > 1$
\begin{equation*}
\Big\| \sum_{n \leq x}{a_{n} n^{-s}} \Big\|_{\mathfrak{D}_{\infty}(X)} \leq\, C \,\log{x}\, \| D\|_{\mathfrak{D}_{\infty}(X)}\, .
\end{equation*}
\end{Theo}

The following result extends \cite[Theorem 3.2]{BayartComposition}
from the scalar case to the vector-valued case
(in the scalar case see also  \cite[Corollary 4]{AOS2014} which proves that  the $1/n^s$ for $1 < p < \infty$ even form a Schauder bases in
$\mathcal{H}_p(\mathbb{C})$).

\begin{Theo}
\label{Theo:PerronPartialSum}
There exists a constant $C > 0$ such that for every $1 \leq p \leq \infty$ and every Dirichlet series $D = \sum_{n}{a_{n} n^{-s}} \in \mathcal{H}_{p}^{+}(X)$  we have that  for each $x > 1$
\begin{equation}
\label{equa:PerronPartialSum}
\Big\| \sum_{n \leq x}{a_{n} n^{-s}} \Big\|_{\mathcal{H}_{p}^{+}(X)} \leq\, C \, \log{x}\,\| D\|_{\mathcal{H}_{p}^{+}(X)}\,.
\end{equation}
\end{Theo}

\begin{proof}
The proof of the case  $1 \leq p < \infty$ is surprinsingly simple if we combine Theorem \ref{old} with the embedding $\mathcal{H}_{p}^{+}(X)$ in $\mathfrak{D}_\infty(\mathcal{H}_{p}(X))$ given in Theorem \ref{Theo:representationHolomorphic2} (which has been already proved for this value of $p$). Indeed, for any Dirichlet series $\sum_{n}{a_{n}n^{-s}}$ in $\mathcal{H}^{+}_p(X)$ we have that
\begin{align*}
\Big\| \sum_{n \leq x}{a_{n} n^{-s}} \Big\|_{\mathcal{H}^{+}_p(X)}
&
= \Big\| \sum_{n \leq x} (a_n n^{-s}) n^{-z} \Big\|_{\mathfrak{D}_{\infty}(\mathcal{H}_{p}(X))}
\\&
\leq C \log{x} \| \sum_{n} (a_n n^{-s}) n^{-z}\|_{\mathfrak{D}_{\infty}(\mathcal{H}_{p}(X))}
\\&
= C \log{x} \| D\|_{\mathcal{H}^{+}_{p}(X)} \,.
\end{align*}
Let us prove now the case $p=+\infty$. If $D \in \mathcal{H}_{\infty}^{+}(X)$, then $ D \in \mathcal{H}_{p}^{+}(X)$ for each $p \in [1, + \infty)$ and so inequality \eqref{equa:PerronPartialSum} is satisfied for every such $p$. By Lemma \ref{Lemm:pnormsLimit}, we can then take limit $p \rightarrow + \infty$ in \eqref{equa:PerronPartialSum} to get the desired conclusion.
\end{proof}

\begin{Coro}
\label{Coro:HpUniformAbcissa}
Let $1 \leq p \leq \infty$ and $D \in \mathcal{H}^{+}_p(X)$. Then for every $\varepsilon > 0$ the partial sums $\sum_{n=1}^{N}{(a_{n} n^{-\varepsilon}) n^{-s}}$ in $\mathcal{H}_p(X)$ converge  to $D_{\varepsilon}$.
\end{Coro}

\begin{proof}
Using Abel's summation formula we have that for every $1 < N < M$
\[ \sum_{n=N}^{M}{\frac{a_n}{n^{\varepsilon + s}}} = \sum_{n=N}^{M-1}{\Big( \sum_{k=1}^{n}{\frac{a_{k}}{k^{s}}}\Big) \Big( \frac{1}{n^{\varepsilon}} - \frac{1}{(n+1)^{\varepsilon}} \Big)} + \Big( \sum_{k=1}^{M}{\frac{a_{k}}{k^{s}}}\Big) \frac{1}{M^{\varepsilon}} - \Big( \sum_{k=1}^{N-1}{\frac{a_k}{k^{s}}} \Big) \frac{1}{N^{\varepsilon}}. \]
Taking norms and using Theorem \ref{Theo:PerronPartialSum} we conclude that
\[
\begin{split}
\left\| \sum_{n=N}^{M}{\frac{a_n}{n^{\varepsilon + s}}} \right\|_{\mathcal{H}^{+}_p(X)} & \leq C \sum_{n=N}^{M-1}{\log{(n)} \Big( \frac{1}{n^{\varepsilon}} - \frac{1}{(n+1)^{\varepsilon}} \Big)} + C \frac{\log{(N-1)}}{N^{\varepsilon}} + C \frac{\log{M}}{M^{\varepsilon}}\\
& \leq C \Big( \sum_{n=N}^{M-1}{\log{(n)} \frac{1}{n^{\varepsilon + 1}}} + \frac{\log{(N-1)}}{N^{\varepsilon}} + \frac{\log{M}}{M^{\varepsilon}}\Big).
\end{split}
\]
It follows that the partial sums $\sum_{n=1}^{N}{\frac{a_n}{n^{\varepsilon}} n^{-s}}$ form a Cauchy sequence in $\mathcal{H}_{p}(X)$, so they converge to a Dirichlet series that must be $D_\varepsilon$ by the uniqueness of the coefficients.
\end{proof}

\begin{Rema}
\label{Rema:proofBiebersau}
The argument in the proof of Corollary \ref{Coro:HpUniformAbcissa} can be adapted to give a proof of Theorem \ref{Biebersau} using Theorem \ref{old}.
\end{Rema}

\noindent
As promised in  the discussion after \eqref{note} we now prove the following equality.

\begin{Coro}
\label{infty=infty}
The following equality holds isometrically:
\[
\mathcal{H}^+_\infty(X) =\mathfrak{D}_\infty(X).
\]
\end{Coro}

\begin{proof}
We already know from \eqref{note} that  we have the inclusion $\mathfrak{D}_\infty(X) \subset \mathcal{H}^+_\infty(X)$. Conversely, take $D = \sum_{n}{a_{n}n^{-s}} \in \mathcal{H}^+_\infty(X)$. By Corollary \ref{Coro:HpUniformAbcissa} and \eqref{EQUA:supremumNormSeries} we have that $\sum_{n}{a_{n}n^{-s}}$ converges uniformly on $\mathbb{C}_{\varepsilon}$ for each $\varepsilon > 0$, and so the limit function $D(s)$ is holomorphic on $\mathbb{C}_{0}$ and satisfies
\[
\|D\|_{\mathfrak{D}_\infty(X)} = \sup_{\Real{(s)} > 0}{\| D(s)\|_{X}} = \sup_{\varepsilon >0} \|D_\varepsilon\|_{\mathcal{H}_{\infty}(X)} = \|D\|_{\mathcal{H}^+_\infty(X)}\,.
\]
Therefore, $D \in \mathfrak{D}_{\infty}(X)$ with norm equal to $\|D\|_{\mathcal{H}^+_\infty(X)}$.
\end{proof}

\section{Operators}
\label{sect:operatorSpaces}

The following characterization of cone absolutely summing operators will be useful and can be found in \cite[Theorem IV.3.3]{ScaheferBanachLattice}.

\begin{Lemm}
\label{Lemm:equivPositAbsSum}
Let $E$ be a Banach lattice and $T \in \mathcal{L}(E,X)$. The following statements are equivalent:
\begin{enumerate}
\item[(i)] $T$ is cone absolutely summing with constant $C$.
\item[(ii)] There is $\xi^\ast \in E^\ast, \| \xi^\ast\| = C$  such that $\|T(\xi)\|_{X} \leq \xi^\ast(|\xi|)$ for every $\xi \in E$.
\end{enumerate}
\end{Lemm}

We turn our attention now on the Banach lattices $E_{p^\ast}(\IDT)$ and the spaces $\Lambda^{+}{\left(E_{p^\ast}(\IDT), X\right)}$ described in the introduction. Recall that every $T$ belonging to $\Lambda^{+}{\left(E_{p^\ast}(\IDT), X\right)}$ is uniquely determined by its Fourier series \eqref{equa:FourierSeriesOperator}, a formal power series in infinitely variables, which in turn corresponds to a Dirichlet series using Bohr transform.

\begin{Theo}
\label{Theo:DirichletOperators}
Let $1 \leq p \leq \infty$. The Bohr transform $\BB$ defines a linear  onto  isometry
 \[
 \Lambda^{+}\left(E_{p^\ast}(\mathbb{T}^\mathbb{N}),X\right) \equiv \mathcal{H}^{+}_{p}(X)\,.
 \]
\end{Theo}

Before we start the proof let us recall the definition of the Poisson kernel in $m$ dimensions.
 For $\omega,z \in \mathbb{T}$ and $0<r<1$
\[
K(\omega, r z)  = \frac{|\omega|^{2} - r^2| z^2|}{|\omega - r z|^2}
=\sum_{n \in \mathbb{Z}}  \omega^{-n} z^n r^{|n|}\,,
\]
and for $\omega,z \in \mathbb{T}^m$ and $r= (r_n)_{n=1}^m \in [0,1[^m$
\begin{align}\label{one}
K(\omega, r z)  = \prod_{n=1}^m K(\omega_n, r_n z_n)
=\sum_{\alpha \in \mathbb{Z}^m}  \omega^{-\alpha} z^\alpha r^{|\alpha|}\,,
\end{align}
where $r^{|\alpha|} := r_1^{|\alpha_1|} \cdot \ldots \cdot r_m^{|\alpha_m|}$.
Moreover, for $1 \leq p \leq \infty$, every function $f \in L_p(\mathbb{T}^m)$ and every $r= (r_n)_{n=1}^m \in [0,1[^m$
the convolution
\[
F_r(z) = \int_{\mathbb{T}^m} f(\omega)  K(\omega, r z)  d\omega\,,\,\, z \in \mathbb{T}^m
\]
again belongs to $L_p(\mathbb{T}^m)$, and
\begin{align}\label{four}
\|F_r\|_{L_p(\mathbb{T}^m)} \leq \|f\|_{L_p(\mathbb{T}^m)}\,;
\end{align}
for $p=1$ we may replace $f(\omega) d \omega$ by a complex Borel measure $m$ on $\mathbb{T}^m$.
We refer to \cite[Chapter 2]{RudinPolydisks} where a complete study of the  properties of the Poisson  kernel
in $m$ dimensions is carried out.

\begin{Lemm} \label{immer schon}
Let $1 \leq p \leq + \infty$. For $m \in \mathbb{N}$ and $f \in E_p(\mathbb{T}^{\mathbb{N}},X)$ let
\[
f_m(\omega) = \int_{\mathbb{T}^{\mathbb{N}}} f(\omega, \tilde{\omega}) d\tilde{\omega}\,,\,\, \omega\in \mathbb{T}^m\,.
\]
Then $f_{m} \in E_p(\mathbb{T}^m,X)$ and moreover
\begin{itemize}
\item[(i)]
$\|f_m\|_{E_p(\mathbb{T}^m,X)} \leq \|f\|_{E_p(\mathbb{T}^{\mathbb{N}},X)}$
\vspace{1mm}
\item[(ii)]
$\displaystyle \lim_{m \rightarrow \infty}f_m = f $ in $E_p(\mathbb{T}^\N,X)$.
\end{itemize}
\end{Lemm}

\begin{proof}
By H\"older's inequality and Fubini's theorem we see that (i) holds. Hence the
linear operators $P_m: E_p(\mathbb{T}^\mathbb{N},X) \rightarrow E_p(\mathbb{T}^\mathbb{N},X)\,, \, P_m(f) = f_m$
are uniformly bounded by $1$, and on the dense subspace of all trigonometric polynomials they
converge to the identity pointwise whenever $m$ tends to infinity. This completes the argument.
\end{proof}

\begin{proof}[Proof of Theorem \ref{Theo:DirichletOperators}]
Our first aim is to show that the Bohr transform $\BB$ is well-defined as an operator from $\Lambda^{+}{\left(E_{p^\ast}, X\right)}$ into  $\mathcal{H}^{+}_{p}(X)$. Fix some operator $T\in \Lambda^{+}{\left(E_{p^\ast}(\mathbb{T}^{\N}), X\right)}$, and define the coefficients
\[
a_n = T(\omega^{-\alpha}) \in X\,\,\, \text{ where }\,\,\, n= \primes^ \alpha.
\]
We intend to show that the Dirichlet series $D = \sum_n a_n n^{-s}  \in \mathcal{H}^{+}_p(X)$.
Let $m \in \mathbb{N}$ and  $\varepsilon >0$, and define
\begin{equation}\label{equa:coordinateRestrictedFunction}
f_{\varepsilon, m}: \mathbb{T}^m \rightarrow X\,,\,\,\,
f_{\varepsilon, m}(z)
=
T\Big(\prod_{n=1}^{m}{K(\omega_{n}, p_{n}^{- \varepsilon}z_n)}\Big)
\end{equation}
(note that $\prod_{n=1}^{m}{K(\omega_{n}, p_{n}^{- \varepsilon}z_n)}$ is a continuous
  function in $\omega$ that belongs to $E_{p^\ast}(\mathbb{T}^\mathbb{N})$). Then for each $z \in \mathbb{T}^m$
  by \eqref{one}
\begin{equation} \label{two}
f_{\varepsilon, m}(z)
=
T\Big(
\sum_{\alpha \in \mathbb{Z}^{m}}{\frac{\omega^{-\alpha}}{(\primes^{|\alpha|})^{\varepsilon}} z^{\alpha}} \Big)
=
\sum_{\alpha \in \N_{0}^{m}}{\frac{T(\omega^{-\alpha})}{(\primes^{\alpha})^{\varepsilon}} z^{\alpha}}\,,
\end{equation}
and hence $f_{\varepsilon, m}$ defines a continuous function on $\overline{\mathbb{D}}^{m}$ which is
holomorphic on $\mathbb{D}^m$.
Next we try to control the norms of these functions in $H_p(\mathbb{T}^m,X)$, and distinguish two  cases:\\

\noindent  The case  $ 1 < p   \leq \infty$:
 Using Lemma \ref{Lemm:equivPositAbsSum}, there is $g \in L_{p}(\IDT) = E_{p^\ast}^\ast(\IDT)$ with $\| g\|_{p} = \| T\|_{\Lambda}$ and $\| T(\xi)\|_X \leq \int_{\IDT}{|\xi| g \: d \sigma}$ for every $\xi \in L_{p^\ast}(\IDT)=E_{p^\ast}(\IDT)$. This applies on \eqref{equa:coordinateRestrictedFunction} which yields for every $z \in \mathbb{T}^m$
\begin{equation}
\label{equa:coordinateRestrictedFunction2}
\| f_{\varepsilon, m}(z) \|_X \leq \int_{\mathbb{T}^{m}}{g_{m}(\omega) \prod_{n=1}^{m}{K(\omega_{n},p_{n}^{-\varepsilon} z_{n})} \: d \omega}\,,
\end{equation}
where for $\omega \in \mathbb{T}^{m}$ we define $g_{m}(\omega) := \int_{\IDT}{g(\omega,\widetilde{\omega}) \: d \widetilde{\omega}}$, which  by H\"older's inequality satisfies $\| g_{m}\|_{p} \leq \| g\|_{p}$.  Now we  apply \eqref{four} and  obtain
\begin{equation}
\label{equa:boundFunctionHp1}
\begin{split}
\|f_{\varepsilon, m} \|_{H_p(\mathbb{T}^m,X)} \leq \int_{\mathbb{T}^{m}}{|g_{m}(\omega)|^{p} \: d \omega} = \| g_m\|_{p}^{p} \leq \| g\|_{p}^{p} = \| T\|_{\Lambda}^{p}.
\end{split}
\end{equation}

  \noindent The case  $  p =1$: Again by Lemma \ref{Lemm:equivPositAbsSum} there exists a non-negative regular Borel measure $\mu \in M(\IDT) = C(\IDT)^* = E_{\infty}(\IDT)^\ast $ with $\| \mu \|=\| T\|_{\Lambda}$ satisfying $\| T(f)\|_X \leq \int_{\IDT}{|f| \: d \mu}$ for all $f \in C(\IDT)$. Using the canonical embedding $C(\mathbb{T}^{m}) \hookrightarrow C(\IDT)$, we have that the restriction of $\mu$ to $C(\mathbb{T}^{m})$ gives a non-negative regular Borel measure on $\mathbb{T}^{m}$ with $\| \mu_{m} \| \leq \| \mu \|$. Hence by  the definition from \eqref{equa:coordinateRestrictedFunction} we for every $z \in \mathbb{T}^m$ have
\[ \| f_{\varepsilon,m}(z) \|_X \leq \int_{\mathbb{T}^{m}}{\prod_{n=1}^{m}{K(\omega_{n}, p_{n}^{- \varepsilon} z_{n})} \: d \mu_{m}}\,, \]
which again by \eqref{four} leads to
\begin{equation}
\label{equa:boundFunctionHp2}
\| f_{\varepsilon,m}\|_{H_{1}(\mathbb{T}^{m}, X)} \leq \| \mu_{m} \| \leq \| \mu \| = \| T\|_{\Lambda}.
\end{equation}

\vspace{2mm}
 \noindent Let us come back to the Dirichlet series $D = \sum_n a_n n^{-s}$. The equations \eqref{equa:boundFunctionHp1} and \eqref{equa:boundFunctionHp2} show that for each $m$ and each $\varepsilon$
 \[ D_{\varepsilon, m}: = \BB(f_{\varepsilon, m}) = \sum_{\alpha \in \N_{0}^m}{\frac{a_{\primes^{\alpha}}}{(\primes^{\alpha})^{\varepsilon}} (\primes^{\alpha})^{-s}}.\]
 satisfies that
 \begin{equation} \label{five}
 \| D_{\varepsilon,m} \|_{\mathcal{H}_{p}(X)} = \| f_{\varepsilon,m}\|_{H_{p}(\mathbb{T}^{m}, X)}\leq \| T\|_{\Lambda}\,.
 \end{equation}
  Our next aim is to show that for every $1 < N \in \N$ and $1 \leq p \leq + \infty$
 \begin{equation}
 \label{equa:perronOperator}
 \left\| \sum_{n=1}^{N}{a_{n} n^{-s}} \right\|_{\mathcal{H}_p(X)} \leq C \: \log{(N)} \: \| T\|_{\Lambda},
 \end{equation}
where $C$ is the universal constant from Theorem \ref{Theo:PerronPartialSum}. Indeed, for such $N$ take $m \in \N$ big enough so that every $1 \leq n  \leq N$ belongs to $\{ \primes^{\alpha} \colon \alpha \in \N_{0}^{m} \}$. Then we deduce from \eqref{five} and the mentioned Theorem \ref{Theo:PerronPartialSum} that for each $\varepsilon > 0$
\[ \left\| \sum_{n=1}^{N}{\frac{a_{n}}{n^{\varepsilon}} n^{-s}} \right\|_{\mathcal{H}_p(X)} \leq C \: \log{(N)} \: \| D_{\varepsilon,m} \|_{\mathcal{H}_{p}(X)} \leq C \: \log{(N)} \: \| T\|_{\Lambda}\,. \]
Taking now limits when $\varepsilon \rightarrow 0^{+}$ in the previous expression, we get \eqref{equa:perronOperator}.

Reasoning as in the proof of Corollary \ref{Coro:HpUniformAbcissa}, it follows from \eqref{equa:perronOperator} that the partial sums of $D_{\varepsilon}$ converge in $\mathcal{H}_p(X)$. Hence  by the uniqueness of the coefficients  we obtain that $D_{\varepsilon}\in\mathcal{H}_p(X)$. We still have to  control the norms of these translations $D_\varepsilon$ of $D$. Using \eqref{five} and Lemma \ref{immer schon} we deduce that
 \begin{align*} \label{six}
  \|D_\varepsilon\|_{\mathcal{H}_p(X)} = \lim_m \|D_{\varepsilon,m}\|_{\mathcal{H}_p(X)}
  = \lim_m \|f_{\varepsilon,m}\|_{H_p(\mathbb{T}^m,X)}
  \leq \|T\|_{\Lambda}\,,
 \end{align*}
This means that
\[ \| D\|_{\mathcal{H}_{p}^{+}(X)} = \sup_{\varepsilon >0}\|D_\varepsilon\|_{\mathfrak{D}_p(X)} \leq \|T\|_{\Lambda}\,, \]
and so $D \in \mathcal{H}_p^+(X)$ with norm $\leq \|T\|_{\Lambda}$.
\vspace{4mm}

\noindent In a second step we have to prove that the Bohr transform $\BB$  in fact defines an isometry onto
from $\Lambda^{+}{\left(E_{p^\ast}(\mathbb{T}^\mathbb{N}), X\right)}$  into $\mathcal{H}_\infty^+(X)$, so let $D(s) = \sum_{n}{a_{n} n^{-s}} \in \mathcal{H}^{+}_{p}(X)$. For each $\varepsilon > 0$, denote $f_{\varepsilon} = \BB^{-1}(D_{\varepsilon}) \in H_{p}(\IDT,X)$. We can define $$T:\spn{\{ \omega^{\alpha} \colon \alpha \in \mathbb{Z}^{(\N)} \}} \rightarrow X$$
as the uniquely determined linear map with $T(\omega^{-\alpha}) = a_{p^{\alpha}}$ for $\alpha \in \N_{0}^{(\N)}$ and $T(\omega^{-\alpha}) = 0$ otherwise. For every trigonometric polynomial $P= \sum_{\alpha \in J}{c_{\alpha} \omega^{\alpha}}$  (here $J$ is some finite set in $\mathbb{Z}^\mathbb{N}$) we have by orthogonality
\begin{equation}
\label{equa:defineOperatorPreimage}
\begin{split}
\| T (P)\|_X & = \Big\| \int_{\IDT}{f_{\varepsilon}(\omega)\,\sum_{\alpha \in J}{c_{\alpha} \omega^{\alpha} (\primes^{|\alpha|})^{\varepsilon}} \: d \omega} \Big\|_X\\
& \leq \int_{\IDT}{ \| f_{\varepsilon}(\omega)\|_X  \Big| \sum_{\alpha \in J}{c_{\alpha} \omega^{\alpha} (\primes^{|\alpha|})^{\varepsilon}} \Big| \: d \omega}\,.
\end{split}
\end{equation}
Moreover, for  $1 < p \leq \infty$ we have that
$$\{ \|f_{\varepsilon}(\cdot)\|_X \colon \varepsilon > 0 \}
\subset
\| D\|_{\mathcal{H}_{p}^{+}(X)} \cdot B_{L_{p}(\IDT)}\,,$$
($B_{L_{p}(\IDT)}$ the closed unit ball of the dual space $L_{p}(\IDT)$ ), while for  $p=1$ we can also view this set as a bounded subset in the dual $M(\IDT) = C(\IDT)^{\ast}$. Taking in any case a weak$^\ast$-cluster point
\[ \xi^{\ast} \in \bigcap_{\delta > 0}{\overline{\{ \|f_{\varepsilon}(\cdot)\|_X \colon 0 < \varepsilon < \delta \}}^{\omega^\ast}}\]
we get $\xi^{\ast} \in E_{p^\ast}(\IDT)^\ast$ with $\| \xi^\ast \| \leq \sup_{\varepsilon > 0}{\| f_{\varepsilon}\|_{L_{p}(\IDT)}} \leq \| D\|_{\mathcal{H}^{+}_p(X)}$ that, by \eqref{equa:defineOperatorPreimage}, satisfies $\left\| T(\psi) \right\|_X \leq \xi^\ast(|\psi|)$ for every trigonometric polynomial $\psi$ on $\mathbb{T}^\mathbb{N}$. The density of these polynomials in $E_p^\ast(\IDT)$ allows us to extend $T$ to a bounded linear operator on $E_p^\ast(\IDT)$, that we for simplicity again denote by $T$, which satisfies  $\| T(\psi)\|_X \leq \xi^{\ast}(|\psi|)$ for each $\psi \in E_{p^\ast}(\IDT)$. Hence $T\in \Lambda^+(E_{p^\ast}(\IDT),X)$ with $\| T\|_{\Lambda} \leq \| \xi^\ast\|_{E_{p^\ast}(\IDT)^\ast} \leq \| D\|_{\mathcal{H}^{+}_{p}(X)}$ by Lemma \ref{Lemm:equivPositAbsSum}.
 This finishes the proof.
\end{proof}

\section{ARNP}
\label{sect:ARNP}

The isometric inclusion $\mathcal{H}_{p}(X) \hookrightarrow \mathcal{H}^{+}_{p}(X)$ ($1 \leq p \leq \infty$) shown in Theorem \ref{Rema1} does not have to be onto. To see an example consider the Dirichlet series
$$D=\sum_{n}{e_{n} n^{-s}}$$
in the Banach space $c_0$ of all null sequences, where $e_n, n \in \mathbb{N}$  as usual stands for the canonical
basis. Then $D \in \mathcal{H}^{+}_{\infty}(c_{0})$ since it obviously converges on $\mathbb{C}_0$ and $\|D(s)\|_{c_0}\leq 1$ for each $s \in \mathbb{C}_0$.
But it does not belong to $\mathcal{H}_{\infty}(c_0) \equiv H_{\infty}(\IDT,c_{0})$ since the fact that $\|e_{n}\|_{c_0} =1$ for each $n$ would contradict the Lemma of Lebesgue-Riemann. This is essentially the same example that one may give to show that the inclusion $H_{\infty}(\mathbb{T},X) \hookrightarrow H_{\infty}(\mathbb{D},X)$ (described in the introduction) is strict. In fact, both inclusions are related.

\begin{Lemm}
\label{Lemm:strongerThanARNP}
If $\mathcal{H}^{+}_{\infty}(X) = \mathcal{H}_{\infty}(X)$, then $H_{\infty}(\mathbb{T},X) \hookrightarrow H_{\infty}(\mathbb{D},X)$ is onto.
\end{Lemm}

\begin{proof}
Let $F \in H_{\infty}(\mathbb{D},X)$, with $F(z) = \sum_{n}{c_{n} z^{n}}$. Since $s \mapsto 2^{-s}$ is a diffeomorphism from $\mathbb{C}_0$ onto $\mathbb{D}$, the Dirichlet series $D(s) = F(2^{-s}) = \sum_{n=1}^{\infty}{c_{n} 2^{-s}}$ ($s \in \mathbb{C}_0$) belongs to $\mathcal{H}^{+}_{\infty}(X)$ and $\| D\|_{\mathcal{H}^{+}_\infty} = \| F\|_\infty$. By hypothesis, $D$ must belong to $\mathcal{H}_{\infty}(X)$, which means that we can find a function $f \in H_{\infty}(\IDT, X)$ with Fourier series $f \equiv \sum_{n \in \N_{0}}{c_{n} \omega_{1}^{n}}$, so it only depends on the first variable $\omega_1$ and is the function of $H_\infty(\mathbb{T},X)$ corresponding to $F$.
\end{proof}

The previous result leads to the question of whether the converse is true. We will see that this is the case. We recall the following Banach space property which characterizes when the isometry $H_{\infty}(\mathbb{T},X) \hookrightarrow H_{\infty}(\mathbb{D},X)$ is onto:

 A complex Banach space $X$ is said to have the Analytic Radon-Nikod\'{y}m Property (ARNP) if every $F \in H_{\infty}(\mathbb{D},X)$ satisfies that
\begin{equation}
\label{EQUA:radialLimit}
\lim_{R \rightarrow 1^{-}}{F(R \omega)} \hspace{3mm} \mbox{ exists for almost every $\omega \in \mathbb{T}$. }
\end{equation}

\vspace{1mm}

This notion was firstly considered and studied by Bukhvalov and Danilevich \cite{Bukhvalov1, BukhvalovDani1}. It is a weaker property than the more known Radon-Nikod\'{y}m Property, and strictly weaker as every Banach lattice without copies of $c_{0}$ has the ARNP \cite[p.105, Theorem 1]{BukhvalovDani1}, which in particular implies that $L^{1}[0,1]$ posseses this property although it does not have the RNP.

It will be convenient for us to work in open half-spaces. We will make use of the well-known Cayley transformation, given by
\[ \varphi(z) = \frac{1+z}{1-z} \hspace{5mm} (z \in \mathbb{D}).\]
This map establishes a diffeomorphism between the open disc $\mathbb{D}$ and the right open halfplane $\mathbb{C}_0$, and also satisfies that $\varphi(e^{iu}) =i \cot{(u/2)}$ for every $u\in \R$. Notice that $g: (0, 2\pi) \rto \R$, $g(u) = \cot{(u/2)}$ is a diffeomorphism.

\begin{Obse}
\label{Obse:StoltzRegions}
If $F \in H_{\infty}(\mathbb{D}, X)$ satisfies \eqref{EQUA:radialLimit}, then one can even show that the following limit exists
\[ \lim_{z \rightarrow \omega, z \in S(\omega, \alpha)}{F(z)} \hspace{3mm} \mbox{  exists for almost every $\omega \in \mathbb{T}$},\]
where
\[ S(\omega, \alpha) := \{ z \in \mathbb{D} \colon |z - \omega| < \alpha (1 - |z|) \} \hspace{3mm} (\alpha > 1). \]
\end{Obse}

The proof of this fact in the scalar-case can be found in \cite{Krantz}, and the arguments can be extrapolated to the vector-valued case with no much difficulty: indeed, from \eqref{EQUA:radialLimit} we get $f \in H_{\infty}(\mathbb{T},X)$ defined as $f(\omega) = \lim_{R \rightarrow 1^{-}}{F(R \omega)}$ a.e. which represents $F$ as in \eqref{equa:PoissonRepresentation}. Following then \cite[section 3]{Krantz}, but working with the maximal function of the scalar-valued function $\omega \mapsto \| f(\omega)\|_{X}$ (which is measurable and essentially bounded), we can adapt the proof of \cite[Theorem 3.2]{Krantz} to get the conclusion for $F$.

\begin{Lemm}
\label{Lemm:horizontalConvergenceARNP}
Let $F: \mathbb{C}_{0} \rightarrow X$ be a bounded and holomorphic function. If $X$ has the ARNP, then $\lim_{\varepsilon \rightarrow 0^+}{F(\varepsilon + it)}$ exists for Lebesgue-almost all $t \in \R$.
\end{Lemm}

\begin{proof}
Let us denote $f = F \circ \varphi$ which is a bounded and holomorphic function on $\mathbb{D}$. We claim that for any $\alpha > 1$, $\varphi^{-1}(\varepsilon + it)$ converges to $\varphi^{-1}(it)$ inside $S(\varphi^{-1}(it), \alpha)$ when $\varepsilon$ goes to zero. Indeed, one can easily check that
\[ \frac{|\varphi^{-1}(\varepsilon + it) - \varphi^{-1}(it)|}{1 - |\varphi^{-1}(\varepsilon + it)|} = \frac{\sqrt{(1 + \varepsilon)^2 + t^2} + \sqrt{(1 - \varepsilon)^2 + t^2}}{2 \sqrt{1 + t^2}}.  \]
Using Observation \ref{Obse:StoltzRegions}, we deduce that if $X$ has the ARNP, then
\[  \lim_{z \rightarrow \omega, z \in S(\omega, \alpha)}{f(z)} \]
exists for every $\omega \in A \subset \mathbb{T}$, being $A$ a set with Haar measure equal to one. Combining this fact with the claim above, we deduce that
\[ \lim_{\varepsilon \rightarrow 0^{+}}{f(\varphi^{-1}(\varepsilon + it))} = \lim_{\varepsilon \rightarrow 0^{+}}{F(\varepsilon + i t)} \]
exists for every $t \in g(A)$. Finally, a diffeomorphism maps measurable sets into measurable sets and null sets into null sets, so $\lambda(g(A)) = 1$.
\end{proof}

\begin{Theo}
\label{Theo:ARNPimpliesOnto}
Given a Banach space $X$, the following are equivalent:
\begin{enumerate}
\item[(i)] $X$ has the ARNP.\vspace{1mm}
\item[(ii)] $\mathcal{H}^{+}_{p}(X) = \mathcal{H}_{p}(X)$ for every $1 \leq p \leq \infty$.\vspace{1mm}
\item[(iii)] $\mathcal{H}^{+}_{p}(X) = \mathcal{H}_{p}(X)$ for some $1 \leq p \leq \infty$.
\end{enumerate}
\end{Theo}

\begin{proof}
We prove first that (i) implies (ii) for $1 \leq p < \infty$. Let $D = \sum_{n}{a_{n} n^{-s}}$ be a Dirichlet series in  $\mathcal{H}^{+}_{p}(X)$. By Theorem \ref{Theo:representationHolomorphic2}, the function
\begin{align}\label{Fatou}
F: \mathbb{C}_0 \rightarrow \mathcal{H}_p(X)\,,\, F(z) = \sum_n \frac{a_n}{n^z} n^{-s}
\end{align}
is  well-defined, holomorphic and bounded. If $X$ has the ARNP, since $1 \leq p < \infty$, we know that $L_{p}(\IDT,X)$ has the ARNP by a result of Dowling \cite{DowlingARNPLebesgue}, and so does its subspace $H_{p}(\IDT,X) = \mathcal{H}_{p}(X)$. We can therefore use \eqref{Fatou} and combine it with Lemma \ref{Lemm:horizontalConvergenceARNP} in order to  conclude that for Lebesgue-almost all $t \in \mathbb{R}$
\[
 \text{$\mathcal{H}_{p}(X)$-$\lim_{\varepsilon \rightarrow 0^+}$} \,\,
 \sum_n \frac{a_n}{n^{ \varepsilon + it}}n^{-s}
 =\sum_n \frac{a_n}{n^{ it}}n^{-s}\,.
 \]
 But then the rotation-invariance of the Lebesgue measure on $\mathbb{T}^\mathbb{N}$ (see again Lemma \ref{Lemm:RotationCoefficients}) yields that as desired
 $D =  \sum_n a_n n^{-s} \in \mathcal{H}_p(X)$.

 This handles the case $1 \leq p < \infty$; let us add the proof of (i) $\Rightarrow$  (ii) for $p= \infty$.
We show that if $\mathcal{H}^{+}_{p_0}(X) = \mathcal{H}_{p_0}(X)$ for some $1 \leq p_0 < \infty$, then $\mathcal{H}^{+}_{\infty}(X) = \mathcal{H}_{\infty}(X)$: Let $D = \sum_n a_n n^{-s} \in \mathcal{H}^{+}_{\infty}(X)$.
We conclude from
Theorem \ref{Theo:DirichletOperators} that there is some  $T \in \Lambda^{+}(L_{1}(\IDT), X)$ such that
$\|D\|_{\mathcal{H}^{+}_\infty(X)} = \|T\|_\Lambda$ and
\[
T(\omega^{-\alpha}) = a_n
\,\,\,\text{ for  each $\alpha \in \mathbb{N}_0^{(\mathbb{N})}$ with  $n = \primes^\alpha$}\,.
\]
On the other hand, since $D \in \mathcal{H}^{+}_{\infty}(X)$ we  have that $D \in \mathcal{H}^{+}_{p_{0}}(X)$, and hence by assumption
$D \in  \mathcal{H}_{p_0}(X)$. In view of the definition of
$\mathcal{H}_{p_0}(X)$ there is  $f \in H_{p_0}(\mathbb{T}^{\mathbb{N}},X)$ such that
$\hat{f}(\alpha) = a_n$ for  each $\alpha \in \mathbb{N}_0^{(\mathbb{N})}$ with  $n = \primes^\alpha$.
The aim is to prove  that $f \in H_{\infty}(\mathbb{T}^{\mathbb{N}},X)$.
We have
\[
T(\omega^{-\alpha}) = \int_{\mathbb{T}^{\mathbb{N}}} f(\omega) \omega^{-\alpha} d \omega
\,\,\,\text{ for  each $\alpha \in \mathbb{N}_0^{(\mathbb{N})}$ },
\]
and hence  for each trigonometric polynomial  $Q \in L_1(\mathbb{T}^{\mathbb{N}})$ with $\|Q\|_1 \leq 1$
\[
 \Big\|\int_{\mathbb{T}^{\mathbb{N}}} f(\omega) Q(\omega) d \omega \Big\|_{X}
 = \big\| T(Q) \big\|_X \leq \|T\| \leq \|T\|_\Lambda = \|D\|_{\mathcal{H}^{+}_\infty(X)}.
\]
Density of the trigonometric polynomials in $L_1(\mathbb{T}^{\mathbb{N}})$ shows that as desired
$D=f \in H_{\infty}(\mathbb{T}^{\mathbb{N}},X) \equiv\mathcal{H}_\infty(X)$.

Finally, it remains to prove that (iii) implies (i). But this we already did in Lemma \ref{Lemm:strongerThanARNP}:
If $\mathcal{H}^{+}_{\infty}(X) = \mathcal{H}_{\infty}(X)$, then $X$ has the ARNP.
\end{proof}

\section{Holomorphy}
\label{sec:holomorphic}
In the scalar case Hardy spaces on the infinite dimensional torus and Hardy
   spaces of Dirichlet series can be identified with  certain spaces of holomorphic functions in infinitely many variables (see e.g. \cite{MultipliersMonomialSets},\cite{ColeGamelin}, and \cite{HardySpaceDirich}). A combination of these ideas with those of the previous sections leads to interesting descriptions
   of $\mathcal{H}_p(X)$ and $\mathcal{H}^+_p(X)$ in terms of infinite dimensional holomorphy.

Let us fix $\SSS \subset c_0$ a \emph{Banach space sequence}, i.e., a Banach space of sequences in $c_{0}$ such that the canonical vectors $e_{n}$, $ n \in \N$ form a $1$-unconditional basis. We consider the open unit ball $B_{c_{0}} \cap \SSS$ as an open subset of $\SSS$, and define
$$H_{\infty}(B_{c_{0}} \cap \SSS,X)$$
to be the Banach space of all holomorphic and bounded functions $F:B_{c_{0}} \cap \SSS \rightarrow X$ endowed with the supremum norm. For $1 \leq p < \infty$,
$$H_{p}(B_{c_0} \cap \SSS, X)$$
stands for the space of all holomorphic functions $F: B_{c_0} \cap \SSS \rightarrow X$ satisfying
\[ \|F \|_{H_{p}(B_{c_0} \cap \SSS, X)}:=\sup_{m \in \N}{\sup_{0 < r < 1}{ \Big( \int_{\mathbb{T}^{m}}{\| F(r z_1, \ldots, r z_m, 0 )\|_{X}^{p} \: dz} \Big)^{1/p} } } < \infty. \]
This is a normed space which in general is not complete (see Remark \ref{complete}). The classical Banach space
$H_{p}(\mathbb{D}^{m},X)$ can be seen as an isometric subspace of $H_{p}(B_{c_0} \cap \SSS, X)$:
\[
H_{p}(\mathbb{D}^{m},X) \hookrightarrow H_{p}(B_{c_{0}} \cap \SSS,X)\,,\,\, f \mapsto f \circ \pi_m \,,
\]
where  $\pi_m : B_{c_0} \cap \SSS \rightarrow \mathbb{D}^{m}$ denotes the projection onto the first $m$ variables.
$H_{p}(\mathbb{D}^{m},X)$ is even $1$-complemented since   for each $F \in H_{p}(B_{c_0} \cap \SSS, X)$ its restriction $F_{m}$  to the first $m$ variables belongs to $H_{p}(\mathbb{D}^{m},X)$ with $\| F_{m}\|_{p} \leq \| F\|_{p}$. This way each such $F$ has an associated monomial series expansion
\[ F \equiv \sum_{\alpha \in \N_{0}^{(\N)}}{c_{\alpha}(F) z^{\alpha}}
 \]
 such that $F_{m}(z) = \sum_{\alpha \in \N_{0}^{m}}{c_{\alpha}(F) z^{\alpha}}$ for each $z \in \mathbb{D}^{m}$ and $m \in \N$.

\begin{Lemm}
\label{Lemm:holomorphicIsideDirichlet}
Let $1 \leq p \leq \infty$ and $\SSS$ as above. The Bohr transform establishes a contraction $\BB:H_{p}(B_{c_0} \cap \SSS,X) \rightarrow \mathcal{H}^{+}_{p}(X)$.
\end{Lemm}

\begin{proof}
Using Theorem \ref{Theo:DirichletOperators}, it is enough to show that for each $F \in H_{p}(B_{c_0} \cap \SSS,X)$ there is an operator $T \in \Lambda^{+}{(E_{p^\ast}, X)}$ with equal  norm and  such that  the Fourier coefficients $T(\omega^{-\alpha})$ equal the monomial coefficients $c_\alpha(F)$. Since we will reason as in the proof of the surjectivity in Theorem \ref{Theo:DirichletOperators}, we omit some details. Given $F \in H_{p}(B_{c_0} \cap \SSS,X)$, we define a linear map
\[ T:\spn{\{ \omega^{\alpha}: \alpha \in \mathbb{Z}^{(\N)} \}} \rightarrow X \mbox{ \: by \: } T(\omega^{-\alpha}) = c_{\alpha}(F), \: \alpha \in \N_{0}^{(\N)}.\]
If $P$ is a trigonometric polynomial in $m$ variables, then by orthogonality
\begin{align*} \label{saubieber}
\|T(P(\omega))\|_X
&
= \Big\| \int_{\mathbb{T}^m}{ P(\omega r^{-1}) F_{m}(r \omega)  \: d \omega} \Big\|_X
\\
&
\leq \int_{\mathbb{T}^m}{\left| P(\omega r^{-1})\right|  \| F_{m}(r \omega)\|_X \: d\omega}.
\end{align*}
Since $\{ \|F_{m}(r \cdot)\|_X \colon 0 < r <1, m \in \N \}$ is bounded in $L_{p}(\IDT) \subset (E_{p^\ast})^\ast$, more precisely contained in $\|F \|_{H_{p}(B_{c_0} \cap \SSS, X)}  B_{L_{p}(\IDT)}$, we can take
\[ \xi^{\ast} \in \bigcap_{0<R<1, M \in \N}{\overline{\{ \| F_{m}(r \cdot) \|_X \colon R < r < 1, M < m \}}^{\omega^\ast}}, \]
which satisfies $\| \xi^\ast\|_{(E_{p^\ast})^\ast} \leq \| F\|_{H_{p}(B_{c_0} \cap \SSS, X)}$ and $\| T(\psi)\|_X \leq \xi^{\ast}(|\psi|)$ for every trigonometric polynomial $\psi$ in infinitely many variables. Finally, the  operator $T$ extends to an element in  $\Lambda^{+}{(E_{p^\ast}, X)}$ with the desired properties
(see the proof of Theorem \ref{Theo:DirichletOperators} for details).
\end{proof}

We finally prove that for $p=\infty$ and $\SSS = B_{c_0}$ as well as $1 \leq p < \infty$ and $\SSS = \ell_2$
the contraction from the preceding lemma is even an onto isometry.
To start let $1 \leq p \leq \infty$ and  $T \in \Lambda^{+}(E_{p^\ast},X)$. Then for every $m \in \N$ we have a well-defined holomorphic function
\begin{equation}
\label{equa:definePartialVariables}
F_{m}(z) =   T\Big( \prod_{n=1}^{m}{K(\omega_{n}, z_{n})} \Big) =\sum_{\alpha \in \N_{0}^{m}}{a_{p^\alpha} z^{\alpha}}\,\,, \,\,\, z \in B_{c_0}.
\end{equation}
Hence we prove  as in the first part of the proof of Theorem \ref{Theo:DirichletOperators} (more precisely see the proof of the equations \eqref{equa:boundFunctionHp1} and \eqref{equa:boundFunctionHp2}) that for every $\SSS$  and every $1 \leq p \leq \infty$
\begin{align}
\label{equa:definePartialVariablesII}
\begin{split}
&
F_{m} \in H_{p}(\mathbb{D}^{m},X)  \hookrightarrow H_{p}(B_{c_0} \cap \SSS,X), X)
\\
&
\|F_{m}\|_{H_{p}(\mathbb{D}^{m},X)}=
\|F_{m}\|_{H_{p}(B_{c_0} \cap \SSS,X)} \leq \| T\|_{\Lambda^+}\,.
\end{split}
\end{align}
\noindent The following theorem exploits this fact in the case  $\SSS = c_0$ and $p =\infty$.

\begin{Theo}
\label{Theo:Holomorphicc0}
The Bohr transform $\BB$ defines an onto isometry from $H_{\infty}(B_{c_0}, X)$ into $\mathcal{H}^{+}_{\infty}(X)$, i.e.,
 $$H_{\infty}(B_{c_0}, X)  \equiv \mathcal{H}^{+}_{\infty}(X)\,. $$
 Moreover, if $D = \BB(F)$, then $F(\primes^{-s}) = \sum_{n}{a_{n} n^{-s}}$ whenever $\Real{(s)} > 0$.
\end{Theo}
\begin{proof}
Lemma \ref{Lemm:holomorphicIsideDirichlet} states that $\BB$  is a well-defined contraction. The argument  that it is an onto  isometry will make use of the following version of Schwartz' lemma which can be easily extended to the multi-dimensional case, and by composing with functionals in a second step to the vector-valued setting: If $f:\mathbb{D}^{m} \rightarrow X$ is a holomorphic function with $f(0) = 0$ and $\| f(z)\|_X < 1$ for each $z\in \mathbb{D}^{m}$, then
for each $z=(z_1, \ldots, z_m)\in \mathbb{D}^{m}$
\begin{equation}
\label{equa:SwartzLemma}
\| f(z)\|_X \leq \max_{1 \leq n \leq m} |z_{n}|;
\end{equation}
indeed, for $X= \mathbb{C}$ and $z_1, \ldots, z_m \in \mathbb{D}$ (not all null) consider
$$g: \mathbb{D} \rightarrow \mathbb{D}\,,\,\, g(\zeta)=f\big(\zeta \cdot \big( (\max_n |z_n|)^{-1}z_1, \ldots, (\max_n |z_n|)^{-1}z_m\big)\big)$$
and apply Schwartz' Lemma to get  $|g(\zeta)| \leq |\zeta|$ for all $\zeta \in \mathbb{D}$. Hence as desired
$|g(\max_n |z_n|)| \leq \max_n |z_n|$.\\

\noindent
Let $D \in \mathcal{H}^{+}_{\infty}(X)$ and $\BB^{-1}D = T \in \Lambda^{+}(L^{1}(\IDT),X)$ given by Theorem \ref{Theo:DirichletOperators}. Define the sequence $(F_{m})_m$ in $H_{\infty}(B_{c_0},X)$  from \eqref{equa:definePartialVariables} which by \eqref{equa:definePartialVariablesII} satisfies
\begin{align} \label{2016}
\| F_{m}\|_{H_{\infty}(B_{c_0},X)} \leq \|T\|_\Lambda \,\,, \,\, m\in \mathbb{N}\,.
\end{align}
We show that $(F_{m})_m$ forms a Cauchy sequence in $H_{\infty}(B_{c_0},X)$ endowed with the compact-open topology:
For  $1 \leq m < M$ and fixed $u \in \mathbb{D}^m,  v\in \mathbb{D}^\mathbb{N}$ we apply  \eqref{equa:SwartzLemma} to the function
$$f:\mathbb{D}^m \rightarrow X\,, \,\,\,z=(z_{m+1}, \ldots, z_{M}) \mapsto F_{m}(u,z,v) - F_{M}(u,z,v)\,,$$
and obtain that for all $z \in B_{c_0}$
\[ \| F_{m}(z) - F_{M}(z)\|_X \leq 2 \|T\|_{\Lambda} \max{\{ |z_{n}| \colon m \leq n \leq M \}}. \]
This gives that $(F_{m})_m$ converges uniformly on each compact set
$$K = \big\{ z \colon |z_n| \leq r_{n}\big\} \subset B_{c_0} \,\,\, \text{ with }
\,\,\, (r_{n})_{n} \in B_{c_{0}}\,.$$
 Since every compact subset of $B_{c_{0}}$ is contained in such a set $K$, we apply
  Montel's theorem and get some  $F \in H_{\infty}(B_{c_0}, X)$ such that $F_m \rightarrow  F$ uniformly on all compact subsets of $B_{c_o}$. Moreover, $F$ is bounded by \eqref{2016} with  $\|F\|_{\infty} \leq \| T\|_{\Lambda}$. To finish the proof of the surjectivity it remains to show that $D = \BB(F)$, i.e., \ $c_{\alpha_0}(F) = a_{n_0}(D)$ whenever $\alpha_0 \in \mathbb{N}_0^{m_0}$
  and $n_0 = \primes^{\alpha_0}$. Indeed, $F_m \rightarrow F$ uniformly on $\mathbb{D}^{m_0}$,
  hence $\lim_m c_{\alpha_0}(F_m) = c_{\alpha_0}(F)$. But by definition $c_{\alpha_0}(F_m) = T(\omega^{-\alpha_0}) = a_{n_0}(D)$  for $m \ge m_0$.

  Finally, we explain why  $F(\primes^{-s}) = \sum_{n=1}^\infty a_{n}\frac{1}{n^s}$ for $\Real{(s)}>0$: Note first that
  both functions $D(s)$ and $F(\primes^{-s})$ are holomorphic on $\mathbb{C}_0$. On the other hand, $D$ converges absolutely for $\Real{(s)}>1$ (we have that $|a_n(D)| \leq \|T\|_\Lambda$ for all $n$), and since
  $D = \BB(F)$ this implies that
  $\sum_{\alpha \in \mathbb{N}_0^{(\mathbb{N})}} |c_\alpha(F) \primes^{-s}| <\infty$.
  Since  $F$ is continuous in $\primes^{-s} \in B_{c_0}$, we obtain
  that $F( \primes^{-s}) = \sum_{\alpha \in \mathbb{N}_0^{(\mathbb{N})}} c_\alpha(F) \primes^{-\alpha s}$,
  and hence
   $
  F( \primes^{-s})
  = \sum_{n=1}^\infty{a_{n}\frac{1}{n^{s}}}
  $
  for $\Real{(s)}>1$.
  But now  the identity theorem shows that  as desired $ F( \primes^{-s}) = D(s)$ on $\mathbb{C}_0$.
  This completes the proof.
  \end{proof}

  \vspace{2mm}

  The final theorem is the $\mathcal{H}_p$-analog of the preceding one.

  \vspace{2mm}

\begin{Theo}
\label{Theo:Holomorphicl2}
Let $1 \leq p < \infty$.
The Bohr transform $\BB$ defines an onto isometry from $H_{p}(B_{c_0} \cap \ell_{2}, X)$ into $\mathcal{H}^{+}_{p}(X)$, i.e.,
 $$H_{p}(B_{c_0} \cap \ell_{2}, X)  \equiv \mathcal{H}^{+}_{p}(X)\,. $$
 Moreover, if $D = \BB(F)$, then $F(\primes^{-s}) = \sum_{n}{a_{n} n^{-s}}$ whenever $\Real{(s)} > 1/2$.
 \end{Theo}
 \begin{proof}
We follow the strategy of the previous  proof. Indeed, our main tool this time is the following inequality, which is proved in \cite{ColeGamelin}, but can be easily extended to the vector-valued setting by composing with functionals: For each $f \in H_{p}(\mathbb{D}^m,X)$ and $z \in \mathbb{D}^{m}$
\begin{equation}
\label{equa:pointwiseEvaluation}
\| f(z)\|_X \leq \| f\|_{H_{p}(\mathbb{D}^m,X)} \prod_{n=1}^{m}{(1 - |z_n|^{2})^{-1/p}}\,.
\end{equation}
Take $D \in \mathcal{H}^{+}_{p}(X)$, and let $\BB^{-1}D = T \in \Lambda^{+}(E_{p^\ast},X)$ be the operator given by Theorem \ref{Theo:DirichletOperators}. For every $m \in \N$, let $F_{m}$ be as in \eqref{equa:definePartialVariables}, which by the discussion above belongs to $H_{p}(\mathbb{D}^{m},X)$ with $\| F_{m}\|_{H_{p}(\mathbb{D}^{m},X)} \leq \|T\|_\Lambda$. Given a sequence $a=(a_{n})_{n \in \N} \in \ell_{2} \cap B_{c_0}$, define the set $K_{a} := {\{z \in c_0\colon |z_{n}| \leq |a_{n}|\}}$ which is compact in $\ell_{2}$ . Then by \eqref{equa:pointwiseEvaluation} we see that
for every $z \in K_{a}$ and every $m$
\[ \| F_{m}(z)\|_X \leq \| T\|_{\Lambda^+} \prod_{n=1}^{\infty}{(1  - |a_{n}|^{2})^{-1/p}}\,. \]
This means that $(F_{m})_{m }$ is a uniformly  bounded sequence of holomorphic $X$-valued functions on $B_{c_0}$. Using \eqref{equa:SwartzLemma} and reasoning  as in the proof of the previous theorem, we  deduce that $(F_{m})_m$ converges uniformly on every set $K_{a} \subset B_{c_0}$ with $a \in \ell_{2} \cap B_{c_0}$. Since every compact subset of $B_{c_{0}} \cap \ell_2$ is contained in some $K_{a}$, this proves (by Montel's theorem) that there is a holomophic function $F:B_{c_{0}} \cap \ell_{2} \rightarrow X$  such  that  $F_{m} \rightarrow F$  uniformly on every compact set of $\ell_{2} \cap B_{c_0}$. From \eqref{equa:definePartialVariablesII} we deduce that
$$\|F\|_{H_{p}(B_{c_0} \cap \ell_{2},X)} = \sup_{m }\|F_{m}\|_{H_{p}(B_{c_0} \cap \ell_{2},X)} \leq \| T\|_{\Lambda^+}\,,$$
which shows (as above) that $F = \BB^{-1}D$ is the desired inverse image of $D$ under the Bohr transform.

Let us finish with the argument for the second statement in the theorem: We have  that $F(\primes^{-s}) = \sum_{n}{a_{n}n^{-s}}$ for $\Real{(s)}>1$ as the Dirichlet series and the monomial expansion of $F$ converge absolutely. Fix some $\varepsilon > 0$, then the function $s \mapsto F(\primes^{-s})$ is bounded on $\mathbb{C}_{1/2 + \varepsilon}$ and extends $D(s)$. By Bohr's Fundamental Theorem \ref{Biebersau} we conclude that $\sum_{n}{a_{n}n^{-s}}$ converges uniformly on $\mathbb{C}_{1/2 + \varepsilon}$. Therefore both holomorphic functions $F(\mathfrak{p}^{-s})$ and $D(s)$ must coincide on $\mathbb{C}_{1/2 + \varepsilon}$.
\end{proof}

\vspace{2mm}

The proof of the previous Theorem \ref{Theo:Holomorphicl2} relies on inequality \eqref{equa:pointwiseEvaluation}, which forces us to work in
$z \in \ell_{2}$ so that the infinite product is convergent. This contrasts with the case $p=\infty$ of Theorem \ref{Theo:Holomorphicc0} in which can be taken $\SSS = c_{0}$, so it is natural to ask if we can consider a bigger Banach space sequence $\SSS$ for some $1 \leq p < \infty$, so that we can identify $H_{p}(B_{c_0} \cap \SSS,X)$ and $\mathcal{H}^{+}_{p}(X)$.

\begin{Rema}
\label{complete}
Let $1 \leq p < \infty$. If $(H_{p}(B_{c_0} \cap \SSS,\mathbb{C}), \| \cdot\|_{p})$ is complete, then $\SSS \subset \ell_{2}$.
\end{Rema}

\begin{proof}
We argue by contradiction. Suppose that there is $a=(a_{n})_{n \in \N} \in \SSS \setminus \ell_{2}$. This means that we can find $(\xi_{n})_{n \in \N} \in \ell_{2}$ such that $\big(\left|\sum_{n=1}^{m}{\xi_{n}a_{n}}\right|\big)_m$ is unbounded. Without loss of generality we can assume that $a \in B_{c_{0}}$. Consider the sequence of trigonometric polynomials $Q_{m}(z) = \sum_{k=1}^{m}{\xi_{k} z_{k}}\,,\, m \in \N$. By Khintchine's inequality, there exists a constant $C > 0$ (independent of $\xi$) such that for every $1 \leq n < m$
\[ \mbox{$\| Q_{m} - Q_{n}\|_{p} \leq C \Big( \sum_{n < k \leq m}{|\xi_{k}|^{2}} \Big)^{1/2}$}. \]
This means that $(Q_{m})_{m }$ is a Cauchy sequence in $H_{p}(B_{c_0} \cap \SSS,X)$. But this space is  assumed to be complete, so it converges to some $F \in H_{p}(B_{c_0} \cap \SSS,X)$ which must be of the form $F \equiv \sum_{k}{\xi_{k} z_{k}}$. By continuity, $F(a)= \lim_{m}{\sum_{n=1}^{m}{\xi_{n} a_{n}}}$, contradicting the choice of $(\xi_{n})_{n \in \N}$.
\end{proof}

A careful look at the proofs of Theorems \ref{Theo:Holomorphicl2} and \ref{Theo:Holomorphicc0} reveals as an  underlying idea what we call {\it a Hilbert's criterion}, i.e.,  a criterion to decide whether a formal power series is generated by a  holomorphic function in $H_{p}(B_{c_{0}} \cap \mathcal{S}_{p},X)$, where  $\mathcal{S}_{p} = c_{0}$ if $p=+\infty$ and $\mathcal{S}_{p} = \ell_{2}$ if $1 \leq p<+\infty$.

\begin{Coro}
\label{Coro:HilbertCriterion}
Let $1 \leq p \leq +\infty$.
Given a family $(c_{\alpha})_{\alpha}$ in $X$, the following assertions are equivalent:
\begin{enumerate}
\item[(i)] There is $F \in H_{p}(B_{c_{0}} \cap \mathcal{S}_{p},X)$ satisfying $c_{\alpha}(F) = c_{\alpha}$ for each $\alpha \in \N_{0}^{(\N)}$.\vspace{2mm}
\item[(ii)] There exists a sequence $(f_{m})$ in  $H_{p}(\mathbb{D}^{m},X)$ such that
\begin{enumerate}
\vspace{1mm}
\item[(ii.1)] $c_{\alpha}(f_{m}) = c_{\alpha}$ for each $\alpha \in \N_{0}^{m}$
\vspace{1mm}
\item[(ii.2)] $\sup_{m \in \N}{\| f_{m}\|_{H_{p}(\mathbb{D}^m,X)}} < + \infty$
\end{enumerate}
\end{enumerate}
\vspace{1mm}
Moreover, in that case $\| F\|_{H_{p}(B_{c_0} \cap \mathcal{S}_{p},X)} = \sup_{m \in \N}{\| f_{m}\|_{H_{p}(\mathbb{D}^m,X)}}$.
\end{Coro}

\begin{proof}
If (i) is satisfied, then, taking for $f_{m}:=F_{m}$ the restriction of $F$ to $\mathbb{D}^{m}\subset B_{c_0} \cap \mathcal{S}_{p}$, we have that (ii) is satisfied with $$\sup_{m \in \N}{\| F_{m}\|_{H_{p}(\mathbb{D}^m,X)}} \leq \| F\|_{H_{p}(B_{c_{0}} \cap \mathcal{S}_{p},X)}\,.$$ To see the converse, notice that the proofs of Theorems \ref{Theo:Holomorphicl2} and \ref{Theo:Holomorphicc0} show precisely that, starting with  $(f_{m})_{m \in \N}$, there is $F \in H_{p}(B_{c_{0}} \cap \mathcal{S}_{p},X)$ whose restriction $F_{m}$ to $\mathbb{D}^{m}$ equals $f_{m}$ and satisfies $\| F\|_{H_{p}(B_{c_{0}} \cap \mathcal{S}_{p},X)} \leq \sup_{m \in \N}{\| F_{m}\|_{H_{p}(\mathbb{D}^m,X)}}$.
\end{proof}

\section{Brother's Riesz and more}
\label{sec:equivalencesARNP}

Recall that if $(\Omega,\mathcal{B})$ is a measure space, then a vector measure $m:\mathcal{B} \rightarrow X$ is of bounded variation if
\[ |m|(\Omega) = \sup_{\pi}{\sum_{A \in \pi}{\|m(A)\|_X}} < +\infty\,, \]
where the supremum is taken over all finite partitions of $\Omega$ into (disjoint) elements of $\mathcal{B}$. If $\Omega$ is a compact space and $\mathcal{B} = \mathcal{B}(\Omega)$ is the Borel $\sigma$-algebra on $\Omega$, consider $M(\mathcal{B}, X)$ to be the space of all regular countably additive vector measures of bounded variation. This is a Banach space when endowed with the variation as norm. Following \cite[p. 380, Theorem 2]{Dinculeanu}, we have the natural isometric identification
\begin{equation}
\label{equa:singerDinculeanu}
\begin{split}
& M(\mathcal{B}, X) \equiv \Lambda(C(\Omega), X)\,, \hspace{3mm} m \mapsto  T_{m}(f) = \int_{\Omega}{f \: d m}\,. \\
\end{split}
\end{equation}
If we restrict ourselves to $\Omega = \IDT$, then we write $M(\mathbb{T}^{\mathbb{N}}, X ) = M(\mathcal{B}(\mathbb{T}^{\mathbb{N}}), X )$,  and the previous identification gives
\begin{equation}
\label{equa:singerDinculeanu2}
M^{+}(\IDT, X )  \equiv \Lambda^{+}(C(\IDT), X)\,,
\end{equation}
where $M^{+}(\IDT, X )$ denotes the subspace of $M(\mathbb{T}^{\mathbb{N}}, X )$ consisting of all vector measures $m$ that are \emph{analytic}, i.e. satisfy
\begin{equation}
\label{equa:analyticMeasureCondition}
\widehat{m}(\alpha) := \int_{\IDT}{\omega^{-\alpha} \: d m}=0 \: \mbox{ for every $\alpha \in \mathbb{Z}^{(\N)} \setminus \N_{0}^{(\N)}$}.
\end{equation}

We now formulate a vector-valued version of the classical Brother's Riesz Theorem in the infinite dimensional torus. The scalar case is a well-known result for topological groups from \cite{HelsonLowdenslager}, and an alternative and different approach for the infinite dimensional torus was recently given in \cite[Corollary 1]{AOS2015}. Recall that the connection between ARNP and the validity of a vector-valued Brother's Riesz Theorem for the one-dimensional torus was firstly noticed by Dowling \cite{DowlingARNP}.

\begin{Theo}
\label{Theo:analyticMeasures}
A Banach space $X$ has the ARNP if and only if for every  $m \in M^{+}(\mathcal{B}(\IDT), X )$ there exists $f \in H_{1}(\IDT,X)$ satisfying
\begin{equation}
\label{more}
m(A) = \int_{A}{f (\omega)\: d\omega}  \: \mbox{ for every $A \in \BB(\IDT)$}\,.
\end{equation}
\end{Theo}

\begin{proof}
By Theorems \ref{Theo:DirichletOperators}  and \ref{Theo:ARNPimpliesOnto}, $X$ has the ARNP if and only if for each operator $T \in \Lambda^{+}(E_{\infty},X)$ there is a (unique) function $f \in H_1(\mathbb{T}^{\mathbb{N}},X)$ such that
$
T(\omega^{-\alpha}) =\hat{f}(\alpha)
$
 for all $\alpha \in \mathbb{N}_0^{(\mathbb{N})}$\,.
But since the trigonometric polynomials are dense in $C(\mathbb{T}^{\mathbb{N}})$, this is equivalent to the fact that for every such $T$ there is a function $f \in H_1(\mathbb{T}^{\mathbb{N}},X)$ satisfying
\begin{equation}
\label{equa:representingOperator}
T(g) = \int_{\mathbb{T}^{\mathbb{N}}} f(\omega) g(\omega) d\omega \,\,\,\text{ for each }\,\,\, g \in C(\mathbb{T}^{\mathbb{N}})\,.
\end{equation}
Using the identification \eqref{equa:singerDinculeanu}, we conclude that $f$ satisfies \eqref{equa:representingOperator} if and only if
 the measure $m$ identified with $T$
satisfies \eqref{equa:analyticMeasureCondition}. This finishes the proof.
\end{proof}

We give now another representation theorem  for spaces of vector-valued functions, relating duality and the ARNP. Consider the embedding
\[ E \otimes X \hookrightarrow \Lambda(E^{\ast}, X)\,, \xi \otimes x \mapsto [\xi^{\ast} \mapsto \xi^{\ast}(\xi) x] \,. \]
If we denote by $E \tilde{\otimes}_{\Lambda} X$ the completion of $E \otimes X$ with the induced norm $\| \cdot \|_{\Lambda}$, then it is shown in \cite[p. 277]{ScaheferBanachLattice} that the following isometric equality holds:
\begin{equation}
\label{equa:dualTensorProduct}
 \Lambda(E, X^{\ast}) = (E \tilde{\otimes}_{\Lambda} X)^\ast\,, \hspace{3mm} T \mapsto [\xi \otimes x \mapsto (T \xi)x]\,.
\end{equation}
Moreover, if we restrict ourselves to the Banach lattices $E_{p}(\Omega), \,1 \leq p \leq +\infty$ defined in the introduction, then there is a canonical isometric isomorphism (see \cite[Examples, p. 274-5]{ScaheferBanachLattice})
\begin{equation}
\label{equa:tensorProductVectorValued}
E_{p}(\Omega) \tilde{\otimes}_{\Lambda} X = E_{p}(\Omega, X).
\end{equation}
Combining \eqref{equa:dualTensorProduct} and \eqref{equa:tensorProductVectorValued}, we obtain for each $1 \leq p \leq \infty$ the isometric identity
\begin{equation}
\label{equa:dualityEp}
\Lambda(E_{p}(\Omega), X^{\ast}) = E_{p}(\Omega, X)^\ast.
\end{equation}
For every $1 \leq p \leq +\infty$, define $$E_{p}(\IDT,X)^{\ast}_{+}$$ as the (weak$^\ast$-closed) subspace of $E_{p}(\IDT,X)^{\ast}$ made of those functionals which vanish on elements of the form $\omega^{-\alpha} \otimes x$ for every $\alpha \in \mathbb{Z}^{(\mathbb{N})} \setminus \mathbb{N}_0^{(\mathbb{N})}$ and $x \in X$.
Restricting the identity from  \eqref{equa:dualityEp} to $\Lambda^{+}(E_{p}(\mathbb{T}^{\mathbb{N}}\, X^{\ast})$, and using Theorem \ref{Theo:DirichletOperators}, we get the following result.

\begin{Prop} \label{Finish}
For every  $1 \leq p \leq \infty$  the Bohr transform establishes a canonical isometric isomorphism
\[ E_{p}(\IDT,X)^{\ast}_{+}  \equiv \mathcal{H}_{p}^{+}(X^{\ast}), \]
which associates to each functional $f$ on $E_{p}(\IDT,X)$ the Dirichlet series $\sum_{n}{a_{n}^{\ast} n^{-s}}$ in $\mathcal{H}_{p}^{+}(X^{\ast})$, whose coefficients $a^{\ast}_{n} \in X^{\ast}$ satisfy
\[ a_{\mathfrak{p}^{\alpha}}^{\ast}(x) = f(\omega^{-\alpha} \otimes x)\, \mbox{ for every $\alpha \in \N_{0}^{(\N)}$ and $x \in X$.}\]
\end{Prop}

\bigskip

If we now again allow the ARNP to enter the stage, then we arrive at the following:

\begin{Theo}
\label{Coro:H1dual}
Given a Banach space $X$, the following assertions are equivalent:
\begin{enumerate}
\item[(i)] $X^{\ast}$ has the ARNP.
\item[(ii)] $\mathcal{H}_{p}(X^{\ast})$ is a dual Banach space for every (resp. some) $1 \leq p \leq +\infty$.
\item[(iii)] $ E_{p}(\IDT, X)^{\ast}_{+} \equiv \mathcal{H}_{p}(X^{\ast}) $  for every (resp. some) $1 \leq p \leq +\infty$.
\end{enumerate}
\end{Theo}

\begin{proof}
The equivalence (i) $\Leftrightarrow$ (iii) is an immediate consequence of Proposition \ref{Finish} and Theorem \ref{Theo:ARNPimpliesOnto}.
(iii) $\Rightarrow$ (ii): $E_{p}(\IDT, X)^{\ast}_{+}$ is a dual, as it is a weak$\ast$-closed subspace of a dual by its very definition. (ii) $\Rightarrow$ (i):
Let us fix $1 \leq p \leq \infty$ such that $\mathcal{H}_{p}(X^{\ast})$ is (isometrically) isomorphic to the dual $Y^\ast$ of a Banach space $Y$. Our aim is to show that $\mathcal{H}_{p}^{+}(X^{\ast}) = \mathcal{H}_{p}(X^{\ast})$. Given $D=\sum_{n}{a_{n}^{\ast}n^{-s}} \in \mathcal{H}_{p}^{+}(X^{\ast})$, we have by definition that $(D_{\varepsilon})_{\varepsilon > 0}$ is a bounded family in $Y^{\ast}=\mathcal{H}_{p}(X^{\ast})$, so we can
(by Alouglu's Theorem) take a weak$^\ast$-cluster point $\tilde{D} = \sum_{n}{\tilde{a}^{\ast}_{n} n^{-s}} \in Y^{\ast}$ of the net $(D_{\varepsilon})_{\varepsilon > 0}$ when $\varepsilon$ goes to zero. We check now that $D = \tilde{D}$, or equivalently, that $a^{\ast}_{N} = \tilde{a}^{\ast}_{N}$ for each $N$. This is a consequence of the following observation: for each $x \in X$ and $N $ the functional $\sum_{n}{a_{n}^{\ast}n^{-s}} \mapsto \langle a_{N}^{\ast},x \rangle$ is an element in $\mathcal{H}_{p}(X^{\ast})^{\ast} = Y^{\ast \ast}$ which is weak$^\ast$-continuous, so it belongs to $Y$. In particular, this yields that
\[ \langle \tilde{a}^{\ast}_{N},x \rangle = \lim_{\varepsilon \rightarrow 0^{+}}{\langle a^{\ast}_{N}N^{-\varepsilon},x \rangle} = \lim_{\varepsilon \rightarrow 0^{+}}{N^{-\varepsilon} \langle a_{N}^{\ast},x \rangle} = \langle a_{N}^{\ast},x \rangle \]
for every $x \in X$ and $N$, finishing the proof.
\end{proof}

We finish with a Hilbert criterion for $H_{p}(\IDT,X)$ (see also Corollary \ref{Coro:HilbertCriterion}). This result  is  an immediate consequence of  Corollary \ref{Coro:HilbertCriterion} combined with the Theorems \ref{Theo:ARNPimpliesOnto} , \ref{Theo:Holomorphicc0} and \ref{Theo:Holomorphicl2}.

\begin{Prop}
Let $X$ have  the ARNP and $1 \leq p \leq \infty$. Then for any family $(c_{\alpha})_{\alpha}$ in $X$ the following assertions are equivalent:
\begin{enumerate}
\item[(i)] There is $f \in H_{p}(\IDT,X)$ satisfying $\hat{f}(\alpha) = c_{\alpha}$ for each $\alpha \in \N_{0}^{(\N)}$.
    \vspace{2mm}
\item[(ii)] There exists a sequence $f_{m} \in H_{p}(\mathbb{T}^{m},X)$ such that
\begin{enumerate}
\item[(ii.1)] $\hat{f}_{m}(\alpha) = c_{\alpha}$ for each $\alpha \in \N_{0}^{m}$,
\vspace{1mm}
\item[(ii.2)] $\sup_{m \in \N}{\| f_{m}\|_{H_{p}(\mathbb{T}^m,X)}} < + \infty$.
\end{enumerate}
\end{enumerate}
\vspace{1mm}
Moreover, in that case $\| f\|_{H_{p}(\IDT, X)} = \sup_{m \in \N}{\| f_{m}\|_{H_{p}(\mathbb{T}^m,X)}}$.
\end{Prop}

\begin{Rema}
We strongly believe that the previous property characterizes the ARNP for $X$.
\end{Rema}

\bigskip

\section*{Acknowledgements}
We are very thankful to Prof. Oscar Blasco who pointed out to us the existence of \cite{BlascoPositivepSumming} and \cite{BlascoBoundaryValues}, which inspired part of this work.

\end{document}